\documentclass[11pt,twoside]{article}
\usepackage{artmods,macros}
\usepackage{enumerate,longtable}
\usepackage{array,amsmath,amssymb}
\usepackage{algorithm}
\usepackage{algpseudocode}
\usepackage{cite}   
\usepackage{enumerate,longtable}
\usepackage{graphicx}

\usepackage[font={scriptsize,tt}]{caption}

\title{PRIMAL AND DUAL ACTIVE-SET METHODS FOR CONVEX QUADRATIC PROGRAMMING}
\author{Anders Forsgren%
        \thanks{Optimization and Systems Theory, Department of
            Mathematics, KTH Royal Institute of Technology, SE-100 44
            Stockholm, Sweden (\texttt{andersf@kth.se}).  Research
            supported by the Swedish Research Council (VR).}
 \and   Philip E. Gill%
        \thanks{Department of Mathematics, University of
        California, San Diego, La Jolla, CA 92093-0112 (\texttt{pgill@ucsd.edu},
        \texttt{elwong@ucsd.edu}).
        Research supported in part by Northrop Grumman Aerospace Systems, and National Science Foundation grants
        DMS-1318480 and DMS-1361421.}
 \and   Elizabeth Wong\addtocounter{footnote}{-1}\footnotemark}

\date{The final publication is available at Springer via \\ http://dx.doi.org/10.1007/s10107-015-0966-2}

\makeatletter
\newcommand{\Newcommand}[2]%
   {\ifx#1\undefined \newcommand{#1}{#2} \else \renewcommand{#1}{#2} \fi}
\ifx\mod\undefined
  \newcommand{\mod}[1]{|#1|}
\else
  \renewcommand{\mod}[1]{|#1|}
\fi
\Newcommand{\Re} {\mathbb{R}}
\Newcommand{\subject}{\mathop{\operator@font{subject\ to}}}

\newcommand{\ph}[1]{\phantom#1}
\newcommand{\m}{\ph-}
\newcommand{\s}[1]{^{\scriptscriptstyle\mathit{#1}}}
\newcommand{\pmat}[1]{\begin{pmatrix}#1\end{pmatrix}}
\newcommand{\thnmtx}[2]{\renewcommand{\arraystretch}{0.95}%
      \left(\nnthiksp\begin{array}{#1}#2\end{array}\nnthiksp\right)}
\newcommand{\st}{:}
\newcommand{\argmin}{\mathop{\operator@font{argmin}}}
\newcommand{\T}{^T\!}

\newcommand{\infnorm}[1]{\norm{#1}_{\infty}\drop}
\newcommand{\disp}{\displaystyle}

\newcommand{\sgap}{\;}
\newcommand{\mgap}{\;\;}
\newcommand{\bgap}{\;\;\;}
\newcommand{\minimize}[1]{{\displaystyle\minim_{#1}}}
\newcommand{\maximize}[1]{{\displaystyle\maxim_{#1}}}
\newcommand{\Bigp}[1]{\bigl(#1\bigr)}
\newcommand{\tmat}[2]{\big(\, #1 \ \ #2 \,\big)}
\newcommand{\tmatt}[3]{\big(\, #1 \ \  #2 \ \  #3\,\big)}
\newcommand{\tmattt}[4]{\big(\, #1 \ \  #2 \ \  #3 \ \  #4\,\big)}
\newcommand{\mat}[2]{\Bigp{\, #1 \mgap #2 \,}}

\newcommand{\mattt}[4]{\Bigp{\; #1 \mgap #2 \mgap #3 \mgap #4\;}}
\newcommand{\maxim}{\mathop{\operator@font{maximize}}}
\newcommand{\minim}{\mathop{\operator@font{minimize}}}
\newcommand{\drop}{^{\null}}
\newcommand{\etal}{et al.}
\newcommand{\Grad}{\nabla\!}
\newcommand{\Hess}{\nabla^2\!}
\newcommand{\half}  {{\textstyle\frac12}}
\newcommand{\inv}{^{-1}}
\newcommand{\norm}[1]{\|#1\|}
\newcommand{\Seq}[1]{\{\hthinsp#1\hthinsp\}}

\newcommand{\words}[1]{\mgap\text{#1}\mgap}
\newcommand{\wordss}[1]{\quad\text{#1}\quad}
\newcommand{\hthinsp}{\mskip  1   mu}    %
\newcommand{\nnthiksp}{\mskip -9.5 mu}
\newcommand{\up}[1]{^{(#1)}}
\newcommand{\submax}{_{\operator@font{max}}}
\newcommand{\A}{_{\scriptscriptstyle A}}

\newcommand{\B}{_{\scriptscriptstyle B}}

\Newcommand{\D}{_{\scriptscriptstyle D}}
\Newcommand{\P}{_{\scriptscriptstyle P}}
\newcommand{\starsymbol}{\ast}
\newcommand{\substar}{_\starsymbol}
\newcommand{\alphastar}{\alpha\substar\drop}
\newcommand{\alphamax}{\alpha\submax}
\providecommand{\varDelta}  {{\mathit\Delta}}
\providecommand{\varPi}     {{\mathit\Pi}}
\newcommand{\Deltait}{\varDelta}
\newcommand{\Piit}{\varPi}
\newcommand{\setB}{\mathcal{B}}
\newcommand{\setI}{\mathcal{I}}
\newcommand{\setN}{\mathcal{N}}
\newcommand{\subU}{_{\scriptscriptstyle U}}

\newcommand{\Ahat}{\skew6\widehat A}

\newcommand{\ghat}{\skew3\widehat g}
\newcommand{\Hhat}{\skew3\widehat H}

\newcommand{\rtilde}{\skew3\widetilde r}
\newcommand{\xbar}{\skew{2.8}\bar x}
\newcommand{\xtilde}{\skew3\widetilde x}
\newcommand{\ytilde}{\skew3\widetilde y}
\newcommand{\zbar}{\skew{2.8}\bar z}
\newcommand{\ztilde}{\skew{2.8}\widetilde z}

\Newcommand{\L}{_{\scriptscriptstyle L}} 
\let\subL\L                              
\newcommand{\N}{_{\scriptscriptstyle N}}
\Newcommand{\R}{_{\scriptscriptstyle R}}
\newcommand{\U}{\subU}
\newcommand{\Bd}{_{\scriptscriptstyle\mathit{B}}\drop}
\newcommand{\Nd}{_{\scriptscriptstyle\mathit{N}}\drop}
\newcommand{\BB}{_{\scriptscriptstyle\mathit{BB}}}
\newcommand{\BBd}{_{\scriptscriptstyle\mathit{BB}}\drop}
\newcommand{\BN}{_{\scriptscriptstyle\mathit{BN}}}
\newcommand{\BNd}{_{\scriptscriptstyle\mathit{BN}}\drop}
\newcommand{\NN}{_{\scriptscriptstyle\mathit{NN}}}
\newcommand{\NNd}{_{\scriptscriptstyle\mathit{NN}}\drop}
\newcommand{\Bl}{_{{\scriptscriptstyle\mathit{Bl}}}}
\newcommand{\Bld}{_{{\scriptscriptstyle\mathit{Bl}}}\drop}
\newcommand{\Bk}{_{{\scriptscriptstyle\mathit{Bk}}}}
\newcommand{\Nk}{_{{\scriptscriptstyle\mathit{Nk}}}}
\newcommand{\Bkd}{_{{\scriptscriptstyle\mathit{Bk}}}\drop}
\newcommand{\Nl}{_{{\scriptscriptstyle\mathit{Nl}}}}
\newcommand{\Nld}{_{{\scriptscriptstyle\mathit{Nl}}}\drop}
\newcommand{\ld}{_l\drop}
\newcommand{\lld}{_{ll}\drop}
\renewcommand{\ll}{_{ll}}
\newcommand{\kkd}{_{kk}\drop}

\newcommand{\kld}{_{kl}\drop}

\newcommand{\optval}{\mathop{\operator@font optval}}
\newcommand{\Primal}{\mbox{\small PQP}}
\newcommand{\Dual}{\mbox{\small DQP}}
\newcommand{\dx}{\Deltait x}
\newcommand{\dxtilde}{\Deltait\xtilde}
\newcommand{\dy}{\Deltait y}
\newcommand{\dytilde}{\Deltait\ytilde}
\newcommand{\dz}{\Deltait z}
\newcommand{\dztilde}{\Deltait\ztilde}
\newcommand{\setBviolated}{\setB^{\scriptscriptstyle<}}
\newcommand{\rBviolated}{\rtilde}
\newcommand{\EXPAND}{{\small EXPAND}}

\newcommand{\KKT} {KKT}

\newcommand{\QP}  {QP}
\providecommand{\CUTEst}{{\small CUTE}st}
\providecommand{\Matlab}{{\sc Matlab}}

\newcommand{\SQOPT}{{\texttt{SQOPT}}}
\newcommand{\QPOPT}{{\texttt{QPOPT}}}
\newcommand{\QPSchur}{\texttt{QPSchur}}
\newcommand{\sqic}{{\texttt{SQIC}}}
\newcommand{\GHAD}{\texttt{GALAHAD}}
\newcommand{\QPA}{\texttt{QPA}}
\newcommand{\PDQP}{\texttt{PDQP}}
\newcommand{\qpOASES}{{\texttt{qpOASES}}}

\newlength{\argwidth}
\newcommand{\mbl}[2]{
\settowidth{\argwidth}{$#1$}
\makebox[\argwidth][l]{$#2$}}%
\newcommand{\mbc}[2]{
\settowidth{\argwidth}{$#1$}
\makebox[\argwidth][c]{$#2$}}%

\newcommand{\blackslug}{\hbox{\hskip 1pt \vrule width 4pt height 6pt depth 1.5pt
  \hskip 1pt}}

\newcommand{\algname}[1]{\mbox{\lowercase{\texttt{#1}}}}
\renewcommand{\algname}[1]{\textrm{#1}}
\algrenewcommand{\algorithmiccomment}[1]{\hfill[#1]}
\newcommand{\agap}{\hspace{1.5ex}}
\newcommand{\STOP}{\mathbf{stop}}

\newcommand{\unb}{\hbox to 0pt{$\null^{u}$}}
\newcommand{\fail}{\hbox to 0pt{$\null^{f}$}}
\newcommand{\infs}{\hbox to 0pt{$\null^{i}$}}
\newcommand{\hstrt}{\rule[-1ex]{0pt}{3.5ex}}
\newcommand{\opt}{_{\mathtt{opt}}}

\newcommand{\epsfeas}{\epsilon_{\mathtt{fea}}}
\newcommand{\Cute}[1]{\hbox{\it\lowercase{#1}\/}}
\renewcommand{\Cute}[1]{\texttt{#1}\/}
\makeatother

\begin{document}

\maketitle

\begin{abstract}
\noindent%
Computational methods are proposed for solving a convex quad\-ratic program
(QP).  Active-set methods are defined for a particular primal and dual
formulation of a QP with general equality constraints and simple lower
bounds on the variables.  In the first part of the paper, two methods are
proposed, one primal and one dual. These methods generate a sequence of
iterates that are feasible with respect to the equality constraints
associated with the optimality conditions of the primal-dual form.  The
primal method maintains feasibility of the primal inequalities while
driving the infeasibilities of the dual inequalities to zero. The dual
method maintains feasibility of the dual inequalities while moving to
satisfy the primal inequalities.
In each of these
methods, the search directions satisfy a KKT system of equations formed
from Hessian and constraint components associated with an appropriate
column basis.  The composition of the basis is specified by an active-set
strategy that guarantees the nonsingularity of each set of KKT equations.
Each of the proposed methods is a conventional active-set method in the
sense that an initial primal- or dual-feasible point is required. In the
second part of the paper, it is shown how the quadratic program may be
solved as a coupled pair of primal and dual quadratic programs created from
the original by simultaneously shifting the simple-bound constraints and
adding a penalty term to the objective function.  Any conventional column
basis may be made optimal for such a primal-dual pair of shifted-penalized
problems. The shifts are then updated using the solution of either the
primal or the dual shifted problem.  An obvious application of this
approach is to solve a shifted dual QP to define an initial feasible point
for the primal (or \emph{vice versa}). The computational performance of
each of the proposed methods is evaluated on a set of convex problems from
the \CUTEst{} test collection.

\end{abstract}

\section{Introduction}
We consider the formulation and analysis of  active-set methods
for a convex quadratic program (\QP\@) of the form
\begin{equation}  \label{eqn-QP-defined}
 \begin{array}{l@{\hspace{10pt}}l}
  \minimize{x\in\Re^n,\,y\in\Re^m} & \half x\T Hx + \half y\T  M y +
  c\T x \\[5pt]
  \subject                         & A x   +  M y = b,  \quad   x\ge 0,
 \end{array}
\end{equation}
where $A$, $b$, $c$, $H$ and $M$ are constant, with $H$ and $M$ symmetric
positive semidefinite. In order to simplify the theoretical discussion, the
inequalities of (\ref{eqn-QP-defined}) involve nonnegativity constraints
only.  However, the methods to be described are easily extended to treat
all forms of linear constraints.  (Numerical results are given for problems
with constraints in the form $x\subL \le x \le x\U$ and $b\subL \le Ax \le b\U$,
for fixed vectors $x\subL$, $x\U$, $b\subL$ and $b\U$.)  If $M =0$, the \QP{}
(\ref{eqn-QP-defined}) is a conventional convex quadratic program with
constraints defined in standard form.  A regularized quadratic program may
be obtained by defining $M=\mu I$ for some small positive parameter $\mu$.
(For applications that require the solution of a regularized \QP{} see,
e.g., \cite{AltG99,SWri98,GilR13}.)

Active-set methods for quadratic programming problems of the form
(\ref{eqn-QP-defined}) solve a sequence of linear equations that involve
the $y$-variables and a subset of the $x$-variables.  Each set of equations
constitutes the optimality conditions associated with an
equality-constrained quadratic subproblem.  The goal is to predict the
optimal active set, i.e., the set of constraints that are satisfied with
equality, at the solution of the problem.  A conventional active-set method
has two phases.  In the first phase, a feasible point is found while
ignoring the objective function; in the second phase, the objective is
minimized while feasibility is maintained. A useful feature of active-set
methods is that they are well-suited for ``warm starts'', where a good
estimate of the optimal active set is used to start the algorithm.  This is
particularly useful in applications where a sequence of quadratic programs
is solved, e.g., in a sequential quadratic programming method or in an ODE-
or PDE-constrained problem with mesh refinement.  Other applications of
active-set methods for quadratic programming include mixed-integer
nonlinear programming, portfolio analysis, structural analysis, and optimal
control.

In Section~\ref{sec-background}, the primal and dual forms of a convex
quadratic program with constraints in standard form are generalized to
include general lower bounds on both the primal and dual variables.  These
problems constitute a primal-dual pair that includes problem
(\ref{eqn-QP-defined}) and its associated dual as a special case. In
Sections~\ref{sec-primal} and \ref{sec-dual}, an active-set method is
proposed for each of the primal and dual forms associated with the
generalized problem of Section~\ref{sec-background}. Both of these methods
provide a sequence of iterates that are feasible with respect to the
equality constraints associated with the optimality conditions of the
primal-dual problem pair.  The primal method maintains feasibility of the
primal inequalities while driving the infeasibilities of the dual
inequalities to zero.  By contrast, the dual method maintains feasibility of
the dual inequalities while moving to satisfy the
primal inequalities.  In each of these methods, the search
directions satisfy a \KKT{} system of equations formed from Hessian and
constraint components associated with an appropriate column basis.  The
composition of the basis is specified by an active-set strategy that
guarantees the nonsingularity of each set of \KKT{} equations.

The methods formulated in Sections~\ref{sec-primal}--\ref{sec-dual}
define conventional active-set methods in the sense that an initial
feasible point is required. In Section~\ref{sec-primal-dual}, a method is
proposed that solves a pair of coupled quadratic programs created from the
original by simultaneously shifting the simple-bound constraints and adding
a penalty term to the objective function.  Any conventional column basis
can be made optimal for such a primal-dual pair of shifted-penalized
problems. The shifts are then updated using the solution of either the
primal or the dual shifted problem. An obvious application of this idea is
to solve a shifted dual \QP{} to define an initial feasible point for the
primal, or \emph{vice-versa}. In addition to the obvious benefit of using
the objective function while getting feasible, this approach provides an
effective method for finding a dual-feasible point when $H$ is positive
semidefinite and $M = 0$.  Finding a dual-feasible point is relatively
straightforward for the strictly convex case, i.e., when $H$ is positive
definite. However, in the general case, the dual constraints for the
phase-one linear program involve entries from $H$ as well as $A$, which
complicates the formulation of the phase-one method considerably.

Finally, in Section~\ref{sec-numerical-results} some numerical experiments
are presented for a simple \Matlab{} implementation of a coupled
primal-dual method applied to a set of convex problems from the \CUTEst{}
test collection \cite{GouOT15,GouOT03}.

There are a number of alternative active-set methods available for solving
a \QP{} with constraints written in the format of problem
(\ref{eqn-QP-defined}).  Broadly speaking, these methods fall into three
classes defined here in the order of increasing generality:
(i) methods for strictly convex quadratic
programming ($H$ symmetric positive definite) \cite{GolI83, GilGMSW84,
  Pow85,Sto86, BarB06}; (ii) methods for convex quadratic programming ($H$
symmetric positive semidefinite) \cite{GilMS06a, Bol97, Huy08, Mae10,
  Won11}; and (iii) methods for general quadratic programming (no
assumptions on $H$ other than symmetry) \cite{Bea67, Fle71, GilM78a,
  BeaB78, BunK80, Hoy86, GilMSW90, GilMSW91, Gou91, Fle00, GouT02, GouT02b,
  Won11, GilW15}. Of the methods specifically designed for convex quadratic
programming, only the methods of Boland~\cite{Bol97} and Wong
\cite[Chapter~4]{Won11} are dual active-set methods.  Some existing
active-set quadratic programming solvers include \QPOPT{} \cite{GilMS95},
\QPSchur{} \cite{BarB06}, \SQOPT{} \cite{GilMS06a}, \sqic{} \cite{GilW15}
and \QPA{} (part of the \GHAD{} software library) \cite{GouOT03b}.

The primal active-set method proposed in Section~\ref{sec-primal} is
motivated by the methods of Fletcher~\cite{Fle71}, Gould~\cite{Gou91}, and
Gill and Wong~\cite{GilW15}, which may be viewed as methods that extend the
properties of the simplex method to general quadratic programming.  At each
iteration, a direction is computed that satisfies a \emph{nonsingular}
system of linear equations based on an estimate of the active set at a
solution. The equations may be written in symmetric form and involve both
the primal and dual variables.  In this context, the purpose of the
active-set strategy is not only to obtain a good estimate of the optimal
active set, but also to ensure that the systems of linear equations that
must be solved at each iteration are nonsingular.  This strategy allows the
application of any convenient linear solver for the computation of the
iterates.  In this paper, these ideas are applied to convex quadratic
programming.  The resulting sequence of iterates is the same as that
generated by an algorithm for general \QP\@, but the structure of the
iteration is different, as is the structure of the linear equations that
must be solved.  Similar ideas are used to formulate the new dual
active-set method proposed in Section~\ref{sec-dual}.

The proposed primal, dual, and combined primal-dual methods use a
``conventional'' active-set approach in the sense that the constraints
remain unchanged during the solution of a given \QP\@.  Alternative
approaches that use a parametric active-set method have been proposed by
Best~\cite{Bes82,Bes96}, Ritter \cite{Rit67,Rit81}, Ferreau, Bock and
Diehl~\cite{FerBD08}, Potschka \etal~\cite{PotKBS10}, and implemented in the
\qpOASES{} package by Ferreau \etal~\cite{FerKPBD14}.  Primal methods based
on the augmented Lagrangian method have been proposed by Delbos and
Gilbert~\cite{DelG05}, Chiche and Gilbert~\cite{ChiG15}, and Gilbert and
Joannopoulos~\cite{GilL15}.  The use of shifts for the bounds have been
suggested by Cartis and Gould~\cite{CarG06} in the context of interior
methods for linear programming. Another class of active-set methods that
are convergent for strictly convex quadratic programs have been
considered by Curtis, Han, and Robinson~\cite{CurHR14}.

\paragraph{Notation and terminology.}
Given vectors $a$ and $b$ with the same dimension, $\min(a,b)$ is a vector
with components $\min(a_i,b_i)$.  The vectors $e$ and $e_j$ denote,
respectively, the column vector of ones and the $j$th column of the
identity matrix $I$. The dimensions of $e$, $e_i$ and $I$ are defined by
the context.  Given vectors $x$ and $y$, the column vector consisting of
the components of $x$ augmented by the components of $y$ is denoted by
$(x,y)$.

\section{Background}  \label{sec-background}
Although the purpose of this paper is the solution of quadratic programs of
the form (\ref{eqn-QP-defined}), for reasons that will become evident in
Section~\ref{sec-primal-dual}, the analysis will focus on the properties of a pair
of problems that may be interpreted as a primal-dual pair of
\QP\@s  associated with problem (\ref{eqn-QP-defined}).
It is assumed throughout that the matrix $\tmat{A}{ M}$ associated with the
equality constraints of problem (\ref{eqn-QP-defined}) has full row
rank. This assumption can be made without loss of generality, as shown in
Proposition~\ref{prop-AMfullrowrank} of the Appendix.  The paper involves a
number of other basic theoretical results that are subsidiary to the main
presentation.  The proofs of these results are relegated to the Appendix.

\subsection{Formulation of the primal and dual problems}
For given constant vectors $q$ and $r$, consider the pair of convex
quadratic programs
\begin{displaymath}
(\Primal_{q,r})\quad
 \begin{array}{ll}
   \minimize{x,y}  & \m\half x\T H x  + \half y\T  M y + c\T x + r\T x\\
   \subject        & \mbl{\null- Hx + A\T y + z = c,}{\null\m Ax +  M y = b,}      \mgap\bgap x \ge -q,
 \end{array}
\]
and
\[
(\Dual_{q,r}) \quad
 \begin{array}{ll}
   \maximize{x,y,z}& -\half x\T H x - \half y\T  M y + b\T y - q\T z\\
   \subject        & \mbl{\null- Hx + A\T y + z = c,}{\null- Hx + A\T y + z = c,} \mgap\bgap z \ge -r.
  \end{array}
\end{displaymath}
The following result gives joint optimality conditions for the triple $(x, y,z)$
such that $(x$, $y)$ is optimal for $(\Primal_{q,r})$, and $(x, y, z)$ is
optimal for $(\Dual_{q,r})$.  If $q$ and $r$ are zero, then $(\Primal_{0,0})$
and $(\Dual_{0,0})$ are the primal and dual problems associated with
(\ref{eqn-QP-defined}).  For arbitrary $q$ and $r$, $(\Primal_{q,r})$ and
$(\Dual_{q,r})$ are essentially the dual of each other, the difference is
only an additive constant in the value of the objective function.

\begin{proposition}\label{prop-regQPopt}
  Let $q$ and $r$ denote constant vectors in $\Re^n$.  If $(x$, $y$,
  $z)$ is a given triple in $\Re^n \times \Re^m \times \Re^n$, then
  $(x$, $y)$ is optimal for $(\Primal_{q,r})$ and $(x$, $y$, $z)$ is
  optimal for $(\Dual_{q,r})$ if and only if
\begin{subequations}\label{eqn-optcond}
\begin{align}
H x +  c - A\T y - z & =   0, \label{eqn-optgradLzero} \\
A x +  M y - b       & =   0, \label{eqn-optfeasprimallin} \\
               x + q &\ge  0, \label{eqn-optfeasprimalbound} \\
               z + r &\ge  0, \label{eqn-optnonnegmult} \\
        (x+q)^T(z+r) & =   0. \label{eqn-optcomp}
\end{align}
\end{subequations}
In addition, the optimal objective values satisfy
$\optval(\Primal_{q,r})-\optval(\Dual_{q,r})=-q\T r$. Finally,
{\rm(\ref{eqn-optcond})} has a solution if and only if the sets
\[
 \big\{ (x,y,z) \st - H x + A\T y + z = c, \sgap z \ge -r \big\} \words{and}
 \big\{ (x,y)   \st   A x +  M   y    = b, \sgap x \ge -q \big\}
\]
are both nonempty.
\end{proposition}
\begin{proof}
Let the vector of Lagrange multipliers for the constraints $Ax + M y - b=0$
be denoted by $\ytilde$. Without loss of generality, the Lagrange
multipliers for the bounds $x + q \ge 0$ of $(\Primal_{q,r})$ may be
written in the form $z + r$, where $r$ is the given fixed vector $r$. With
these definitions, a Lagrangian function $L(x,y,\ytilde,z)$ associated with
$(\Primal_{q,r})$ is given by
\begin{multline*}
 L(x,y,\ytilde,z) = \half x\T H x + (c + r)\T x + \half y\T  M y -
                          \ytilde\T (Ax +  M y - b) \\ - (z + r)\T (x+q).
\end{multline*}
Stationarity of the Lagrangian with respect to $x$ and $y$ implies
that
\begin{subequations}\label{eqn-gradLzero}
\begin{align}
Hx + c + r - A\T\ytilde - z - r  &=  Hx + c - A\T \ytilde - z  = 0, \label{eqn-gradLzeroI}\\
                 M y -  M\ytilde &=  0.                             \label{eqn-gradLzeroII}
\end{align}
\end{subequations}
The optimality conditions for $(\Primal_{q,r})$ are then given by: (i) the
feasibility conditions (\ref{eqn-optfeasprimallin}) and
(\ref{eqn-optfeasprimalbound}); (ii) the nonnegativity conditions
(\ref{eqn-optnonnegmult}) for the multipliers associated with the bounds $x+q\ge 0$;
(iii) the stationarity conditions (\ref{eqn-gradLzero}); and (iv) the
complementarity conditions (\ref{eqn-optcomp}). The vector $y$ appears only
in the term $M y$ of (\ref{eqn-optfeasprimallin}) and
(\ref{eqn-gradLzeroII}). In addition, (\ref{eqn-gradLzeroII}) implies that
$M y =  M \ytilde$, in which case we may choose $y=\ytilde$. This common value
of $y$ and $\ytilde$ must satisfy (\ref{eqn-gradLzeroI}), which is then
equivalent to (\ref{eqn-optgradLzero}). The optimality
conditions (\ref{eqn-optcond}) for $(\Primal_{q,r})$ follow directly.

With the substitution $\ytilde = y$, the expression for the Lagrangian may
be rearranged so that
\begin{equation}\label{eqn-Lyy}
  L(x,y,y,z)
    = -\half x\T H x - \half y\T  M y + b\T y - q\T z + (Hx + c - A\T y -z)\T x - q\T r.
\end{equation}
Taking into account (\ref{eqn-gradLzero}) for $y=\ytilde$, the dual
objective is given by (\ref{eqn-Lyy}) as $-\half x\T H x - \half y\T  M y +
b\T y - q\T z-q\T r$, and the dual constraints are $Hx + c - A\T y - z = 0$
and $z+r\ge 0$. It follows that $(\Dual_{q,r})$ is equivalent to the dual
of $(\Primal_{q,r})$, the only difference is the constant term $-q\T r$ in
the objective, which is a consequence of the shift $z+r$ in the dual
variables. Consequently, strong duality for convex quadratic programming
implies $\optval(\Primal_{q,r})-\optval(\Dual_{q,r})=-q\T r$. In addition,
the variables $x$, $y$ and $z$ satisfying (\ref{eqn-optcond}) are feasible
for $(\Primal_{q,r})$ and $(\Dual_{q,r})$ with the difference in the
objective function value being $-q\T r$. It follows that $(x,y,z)$ is
optimal for $(\Dual_{q,r})$ as well as $(\Primal_{q,r})$. Finally,
feasibility of both $(\Primal_{q,r})$ and $(\Dual_{q,r})$ is both necessary
and sufficient for the existence of optimal solutions.
\end{proof}

\subsection{Optimality conditions and the KKT equations}
The proposed methods are based on maintaining index sets $\setB$ and
$\setN$ that define a partition of the index set $\setI = \{1$, $2$, \dots,
$n\}$, i.e., $\setI = \setB \cup \setN$ with $\setB\cap \setN = \emptyset$.
Following standard terminology, we refer to the subvectors $x\B$ and $x\N$
associated with an arbitrary $x$ as the basic and nonbasic variables,
respectively. The crucial feature of $\setB$ is that it defines a unique
solution $(x,y,z)$ to the equations
\begin{equation} \label{eqn-B-N-lin}
\begin{alignedat}{2}
  H x + c - A^T y - z   &= 0, \quad & x\N + q\N &=  0, \\
  A x +       M y - b   &= 0, \quad & z\B + r\B &=  0.
\end{alignedat}
\end{equation}
For the symmetric Hessian $H$, the matrices $H\BB$ and $H\NN$ denote the
subset of rows and columns of $H$ associated with the sets $\setB$ and
$\setN$, respectively. The unsymmetric matrix of components $h_{ij}$ with
$i\in\setB$ and $j\in \setN$ will be denoted by $H\BN$.  Similarly, $A\B$
and $A\N$ denote the matrices of columns of $A$ associated with $\setB$ and
$\setN$ respectively. With this notation, the equations (\ref{eqn-B-N-lin})
may be written in partitioned form as
\[
\begin{alignedat}{2}
  H\BB   x\B  + H\BN x\N  + c\B - A\B^T y - z\Bd      &= 0, & \quad x\N + q\N &=  0,\\
  H\BN^T x\Bd + H\NN x\N  + c\N - A\N^T y       - z\Nd&= 0, &       z\B + r\B &=  0,\\
  A\B    x\B  + A\N  x\N  +  M          y       - b   &= 0.
\end{alignedat}
\]
Eliminating $x\N$ and $z\B$ from these equations using the equalities $x\N
+ q\N = 0$ and $z\B + r\B = 0$ yields the symmetric equations
\begin{equation}\label{eqn-xBy}
   \pmat{ H\BB & A\B^T \\
          A\B  & -M      }
   \pmat{ \m x\B \\
           - y }
 = \pmat{  H\BN q\Nd - c\B - r\B \\
           A\N  q\N  + b           }
\end{equation}
for $x\B$ and $y$. It follows that (\ref{eqn-B-N-lin}) has a unique
solution if and only if (\ref{eqn-xBy}) has a unique solution. Therefore,
if $\setB$ is chosen to ensure that (\ref{eqn-B-N-lin}) has a unique
solution, it must follow from (\ref{eqn-xBy}) that the matrix $K\B$ such
that
\begin{equation}  \label{eqn-submin-KKT}
 K\B = \pmat{ H\BB & A\B^T \\
              A\B  & -M      }
\end{equation}
is nonsingular.  Once $x\B$ and $y$ have been computed, the $z\N$-variables
are given by
\begin{equation}\label{eqn-zN}
 z\N = H\BN^T x\Bd - H\NNd q\Nd + c\Nd - A\N^T y.
\end{equation}
As in Gill and Wong~\cite{GilW15}, any set $\setB$ such that $K\B$ is
nonsingular is referred to as a \emph{second-order consistent basis}.
Methods that impose restrictions on the eigenvalues of $K\B$ are known
as inertia-controlling methods.  (For a description of
inertia-controlling methods for general quadratic programming, see,
e.g., Gill \etal~\cite{GilMSW91}, and Gill and Wong~\cite{GilW15}.)

The two methods proposed in this paper, one primal, one dual, generate a
sequence of iterates that satisfy the equations (\ref{eqn-B-N-lin}) for
some partition $\setB$ and $\setN$. If the conditions (\ref{eqn-B-N-lin})
are satisfied, the additional requirement for fulfilling the optimality
conditions of Proposition~\ref{prop-regQPopt} are $x\B + q\B \ge 0$ and
$z\N + r\N \ge 0$. The primal method of Section~\ref{sec-primal} imposes
the restriction that $x\B + q\B \ge 0$, which implies that the sequence of
iterates is primal feasible. In this case the method terminates when
$z\B+r\B\ge 0$ is satisfied. Conversely, the dual method of
Section~\ref{sec-dual} imposes dual feasibility by means of the bounds $z\N
+ r\N \ge 0$ and terminates when $x\B+q\B\ge 0$.

In both methods, an iteration starts and ends with a second-order
consistent basis, and comprises one or more \emph{subiterations}.  In each
subiteration an index $l$ and index sets $\setB$ and $\setN$ are known such
that $\setB \cup \{ l \} \cup \setN = \{1$, $2$, \dots, $n\}$. This
partition defines a search direction $(\dx,\dy,\dz)$ that satisfies the
identities
\begin{equation}\label{eqn-dxdydz}
\begin{alignedat}{2}
  H \dx - A\T \dy -   \dz &= 0, \qquad&   \dx\N &= 0, \\
          A   \dx +  M\dy &= 0, \qquad&   \dz\B &= 0.
\end{alignedat}
\end{equation}
As $l\not\in\setB$ and $l\not\in\setN$, these conditions imply that neither
$\dx_l$ nor $\dz_l$ are restricted to be zero.  The conditions $\dx\N = 0$
and $\dz\B = 0$ imply that (\ref{eqn-dxdydz}) may be expressed in the
partitioned-matrix form
\[
  \pmat{ h\lld  & h\Bl^T & a_l^T & \mbc{a_l^T}{1} &              \\[1pt]
         h\Bl   & H\BB   & A\B^T &                &              \\[1pt]
         h\Nld  & H\BN^T & A\N^T &                &\mbc{A\N^T}{I}\\[1pt]
         a_l    & A\B    & -M    &                &                }
  \pmat{\m \dx_l \\
        \m \dx\B \\
         - \dy   \\
         - \dz_l \\
         - \dz\N   }
= \pmat{ 0 \\[1pt] 0 \\[1pt] 0 \\[1pt] 0 },
\]
where $h\ll$ denotes the $l$th diagonal of $H$, and the column vectors $h\Bl$ and $h\Nl$ denote
the column vectors of elements $h_{il}$ and $h_{jl}$ with $i\in\setB$, and
$j\in\setN$, respectively.  It follows that $\dx_l$, $\dx\B$, $\dy$ and
$\dz_l$ satisfy the homogeneous equations
\begin{subequations}\label{eqn-hom}
\begin{gather}
  \pmat{ h\lld  & h\Bl^T & a_l^T & \mbc{a_l^T}{1}\\
         h\Bld  & H\BBd  & A\B^T &               \\
         a_l    & A\B    &  -M                     }
             \pmat{ \m \dx_l \\
                    \m \dx\B \\
                     - \dy   \\
                     - \dz_l   }
          = \pmat{ 0 \\ 0 \\ 0 },                     \label{eqn-homa} \\
\intertext{and $\dz\N$ is given by}
  \dz\N   = h\Nld \dx\ld + H\BN^T \dx\Bd - A\N^T \dy. \label{eqn-homb}
\end{gather}
\end{subequations}
The properties of these equations are established in the next subsection.

\subsection{The linear algebra framework}
This section establishes the linear algebra framework that serves to
emphasize the underlying symmetry between the primal and dual methods.  It
is shown that the search direction for the primal and the dual method is a
nonzero solution of the homogeneous equations (\ref{eqn-homa}), i.e., every
direction is a nontrivial null vector of the matrix of (\ref{eqn-homa}). In
particular, it is shown that the null-space of (\ref{eqn-homa}) has
dimension one, which implies that the solution of (\ref{eqn-homa}) is
unique up to a scalar multiple.  The length of the direction is then
completely determined by fixing either $\dx_l =1$ or $\dz_l=1$.  The choice
of which component to fix depends on whether or not the corresponding
component in a null vector of (\ref{eqn-homa}) is nonzero. The conditions
are stated precisely in Propositions~\ref{prop-KBnonsing} and
\ref{prop-Klnonsing} below.

The first result shows that the components $\dx_l$ and $\dz_l$ of any
direction $(\dx$, $\dy$, $\dz)$  satisfying the identities
(\ref{eqn-dxdydz}) must be such that $\dx_l\dz_l\ge 0$.

\begin{proposition}\label{prop-dotproduct}
If the vector $(\dx,\dy,\dz)$ satisfies the identities
\begin{align*}
 H \dx - A\T \dy -   \dz &= 0, \\
         A   \dx +  M\dy &= 0,
\end{align*}
then $\dx^T\dz = \dx\T H \dx + \dy\T M \dy \ge 0$.  Moreover, given an
index $l$ and index sets $\setB$ and $\setN$ such that
$\setB\cup\{l\}\cup\setN = \{1$, $2$, \dots, $n\}$ with $\dx\N =0$ and
$\dz\B = 0$, then $\dx_l\dz_l =\dx\T H \dx + \dy\T  M \dy \ge 0$.
\end{proposition}
\begin{proof}
Premultiplying the first identity by $\dx^T$ and the second by $\dy^T$ gives
\[
 \dx\T H \dx - \dx\T A\T\dy - \dx\T\dz   = 0,  \wordss{and}
               \dy\T A  \dx + \dy\T  M\dy = 0.
\]
Eliminating the term $\dx\T A\T\dy$ gives $\dx\T H \dx + \dy\T  M \dy =
\dx\T\dz$.  By definition, $H$ and $M$ are symmetric positive semidefinite,
which gives $\dx\T\dz\ge 0$. In particular, if $\setB\cup\{l\}\cup\setN =
\{ 1$, $2$, \dots, $n \}$, with $\dx\N = 0$ and $\dz\B=0$, it must hold
that $\dx\T\dz =\dx_l\dz_l\ge 0$.
\end{proof}

The set of vectors $(\dx_l$, $\dx\B$, $\dy$, $\dz_l$, $\dz\N)$ satisfying
the equations (\ref{eqn-hom}) is completely characterized by the properties
of the matrices $K\B$ and $K_l$ such that
\begin{equation}  \label{eqn-KB-Kl-defined}
 K\B = \pmat{ H\BBd  & A\B^T \\
              A\B    &  -M      }  \wordss{and}
 K_l = \pmat{ h\lld  & h\Bl^T & a_l^T \\
              h\Bld  & H\BBd  & A\B^T \\
              a_l    & A\B    &  -M     }.
\end{equation}
The properties are summarized by the results of the following two propositions.

\begin{proposition}\label{prop-KBnonsing}
  Assume that $K\B$ is nonsingular. Let $\dx_l$ be a given nonnegative
  scalar.
 \begin{enumerate}[\rm1.]
 \item\label{prop-KBnonsing:dx-zero} If $\dx_l = 0$, then the only solution of
  {\rm(\ref{eqn-hom})} is zero, i.e., $\dx\B=0$, $\dy=0$, $\dz_l=0$
  and $\dz\N=0$.

 \item\label{prop-KBnonsing:dx-pos}  If $\dx_l > 0$, then the quantities $\dx\B$, $\dy$,
  $\dz_l$ and $\dz\N$ of {\rm(\ref{eqn-hom})} are unique and satisfy
  the equations
\begin{equation}\label{eqn-KB}
\begin{aligned}
 \pmat{ H\BBd & A\B^T \\
         A\B  & -M      }
 \pmat{ \m \dx\B \\
         - \dy     }
       &= -\pmat{ h\Bl \\
                  a_l    } \dx_l,                  \\
 \dz_l &= h\lld \dx\ld + h\Bl^T \dx\Bd - a_l\T \dy,\\
 \dz\N &= h\Nld \dx\ld + H\BN^T \dx\Bd - A\N^T \dy.  
\end{aligned}
\end{equation}
Moreover, either
\begin{enumerate}[\rm(i)]
 \item\label{prop-KBnonsing:dz-pos} $K_l$ is nonsingular and $\dz_l>0$, or

 \item\label{prop-KBnonsing:dz-zero} $K_l$ is singular and $\dz_l=0$, in
   which case it holds that $\dy=0$, $\dz\N=0$, and the multiplicity of the
   zero eigenvalue of $K_l$ is one, with corresponding eigenvector
   $(\dx_l,\dx\B,0)$.
\end{enumerate}
\end{enumerate}
\end{proposition}
\begin{proof}
Proposition~\ref{prop-dotproduct} implies that $\dz_l\ge 0$ if $\dx_l>0$,
which implies that the statement of the proposition includes all possible
values of $\dz_l$. The second and third blocks of the equations
(\ref{eqn-homa}) imply that
\begin{equation}\label{eqn-dxzero}
   \pmat{ h\Bl \\
          a_l    } \dx_l
 + \pmat{ H\BB   & A\B^T \\
          A\B    &  -M     }
   \pmat{ \m \dx\B\\
           - \dy    }
 = \pmat{ 0 \\
          0   }.
\end{equation}
As $K\B$ is nonsingular by assumption, the vectors $\dx\B$ and $\dy$ must
constitute the unique solution of (\ref{eqn-dxzero}) for a given value
of $\dx_l$.  Furthermore, given $\dx\B$ and $\dy$, the quantities $\dz_l$
and $\dz\N$ of (\ref{eqn-KB}) are also uniquely defined.  The specific
value $\dx_l=0$, gives $\dx\B=0$ and $\dy=0$, so that $\dz_l=0$ and
$\dz\N=0$.  It follows that $\dx_l$ must be nonzero for at least one of the
vectors $\dx\B$, $\dy$, $\dz_l$ or $\dz\N$ to be nonzero.

Next it is shown that if $\dx_l>0$, then either
(\ref{prop-KBnonsing:dz-pos}) or (\ref{prop-KBnonsing:dz-zero}) must hold.
For (\ref{prop-KBnonsing:dz-pos}), it is necessary to show that if $\dx_l>0$ and $K_l$ is
nonsingular, then $\dz_l >0$.  If $K_l$ is nonsingular, the homogeneous
equations (\ref{eqn-homa}) may be written in the form
\begin{align}\label{eqn-Klnonsing}
 \pmat{ h\lld  & h\Bl^T & a_l^T      \\
        h\Bl   & H\BB   & A\B^T      \\
        a_l    & A\B    & -M          }
             \pmat{ \m \dx_l \\
                    \m \dx\B \\
                     - \dy    }
   &= \pmat{ 1 \\ 0 \\ 0 } \dz_l,
\end{align}
which implies that $\dx_l$, $\dx\B$ and $\dy$ are unique for a given value
of $\dz_l$. In particular, if $\dz_l =0$ then $\dx_l = 0$, which would
contradict the assumption that $\dx_l >0$. If follows that $\dz_l$ must be
nonzero.  Finally, Proposition~\ref{prop-dotproduct} implies that if $\dz_l$ is
nonzero and $\dx_l>0$, then $\dz_l>0$ as required.

For the first part of (\ref{prop-KBnonsing:dz-zero}), it must be shown that
if $K_l$ is singular, then $\dz_l=0$.  If $K_l$ is singular, it must have a
nontrivial null vector $(p_l$, $p\B$, $-u)$. Moreover, every null vector
must have a nonzero $p_l$, because otherwise $(p\B$, $-u)$ would be a
nontrivial null vector of $K\B$, which contradicts the assumption that
$K\B$ is nonsingular. A fixed value of $p_l$ uniquely defines $p\B$ and
$u$, which indicates that the multiplicity of the zero eigenvalue must be
one. A simple substitution shows that $(p_l$, $p\B$, $-u$, $v_l)$ is a
nontrivial solution of the homogeneous equation (\ref{eqn-homa}) such that
$v_l= 0$.  As the subspace of vectors satisfying (\ref{eqn-homa}) is of
dimension one, it follows that every solution is unique up to a scalar
multiple.  Given the properties of the known solution $(p_l$, $p\B$, $-u$,
$0)$, it follows that every solution $(\dx_l$, $\dx\B$, $-\dy$, $-\dz_l)$
of (\ref{eqn-homa}) is an eigenvector associated with the zero eigenvalue
of $K_l$, with $\dz_l = 0$.

For the second part of (\ref{prop-KBnonsing:dz-zero}), if $\dz_l = 0$, the
homogeneous equations (\ref{eqn-homa}) become
\begin{equation}\label{eqn-Klsing}
\pmat{   h\lld  & h\Bl^T & a_l^T      \\
         h\Bl   & H\BB   & A\B^T      \\
         a_l    & A\B    &  -M           }
             \pmat{ \m \dx_l \\
                    \m \dx\B \\
                     - \dy }
             = \pmat{ 0 \\ 0 \\ 0 }.
\end{equation}
As $K_l$ is singular in (\ref{eqn-Klsing}), Proposition~\ref{propA-nonsing}
of the Appendix implies that
\begin{equation}\label{eqn-KlsingII}
 \pmat{   h\lld  & h\Bl^T      \\
          h\Bl   & H\BB        \\
          a_l    & A\B           }
       \pmat{ \dx_l \\
              \dx\B }
 = \pmat{ 0 \\ 0 \\ 0 },           \wordss{and}
   \pmat{   a_l^T   \\
            A\B^T   \\
           -M         } \dy
 = \pmat{ 0 \\ 0 \\ 0 }.
\end{equation}
The nonsingularity of $K\B$ implies that $\tmat{ A\B }{ -M }$ has full row
rank, in which case the second equation of (\ref{eqn-KlsingII}) gives
$\dy=0$.  It follows that every eigenvector of $K_l$ associated with the
zero eigenvalue has the form $(\dx_l$, $\dx\B$, $0)$.  It remains to show
that $\dz\N=0$. If Proposition~\ref{propA-diag} of the Appendix is applied
to the first equation of (\ref{eqn-KlsingII}), then it must hold that
\[
 \pmat{   h\lld  & h\Bl^T  \\
          h\Bl   & H\BB    \\
          h\Nl   & H\BN^T   }
       \pmat{ \dx_l \\
              \dx\B   }
 = \pmat{ 0 \\ 0 \\ 0 }.
\]
It follows from the definition of $\dz\N$ in (\ref{eqn-KB}) that $\dz\Nd
= h\Nld\dx\ld + H\BN^T \dx\Bd - A\N^T \dy = 0$, which completes the proof.
\end{proof}

\begin{proposition}\label{prop-Klnonsing}
  Assume that $K_l$ is nonsingular. Let $\dz_l$ be a given nonnegative
  scalar.
 \begin{enumerate}[\rm1.]
 \item\label{prop-Klnonsing:dz-zero}
  If $\dz_l = 0$, then the only solution of
  {\rm(\ref{eqn-hom})} is zero, i.e., $\dx_l=0$, $\dx\B=0$, $\dy=0$
  and $\dz\N=0$.

 \item\label{prop-Klnonsing:dz-pos}
  If $\dz_l > 0$, then the quantities $\dx_l$,
  $\dx\B$, $\dy$ and $\dz\N$ of {\rm(\ref{eqn-hom})} are unique and
  satisfy the equations
\begin{subequations}\label{eqn-Klsolve}
\begin{align}
   \pmat{ h\lld & h\Bl^T & a_l^T \\
          h\Bld & H\BBd  & A\B^T \\
          a_l   & A\B    & -M      }
   \thnmtx{l}{ \m \dx_l \\
               \m \dx\B \\
               -  \dy }
 &= \pmat{ 1 \\
           0 \\
           0   } \dz_l,                      \label{eqn-KBmain} \\
 \dz\N
 &= H\Nld \dx\ld + H\BN^T \dx\Bd - A\N^T \dy. \label{eqn-dzN}
\end{align}
\end{subequations}
Moreover, either
\begin{enumerate}[\rm(i)]
 \item\label{prop-Klnonsing:dx-pos} $K\B$ is nonsingular and $\dx_l>0$, or

 \item\label{prop-Klnonsing:dx-zero} $K\B$ is singular and $\dx_l=0$, in which case, it holds that
   $\dx\B=0$ and the multiplicity of the zero eigenvalue of $K\B$ is one,
   with corresponding eigenvector $(0,\dy)$.
\end{enumerate}
\end{enumerate}
\end{proposition}
\begin{proof}
In Proposition~\ref{prop-dotproduct} it is established that $\dx_l\ge 0$ if
$\dz_l>0$, which implies that the statement of the proposition includes all
possible values of $\dx_l$.

It follows from (\ref{eqn-homa}) that $\dx_l$, $\dx\B$, and $\dy$ must
satisfy the equations
\begin{equation}\label{eqn-dzzero}
 \pmat{ h\lld  & h\Bl^T & a_l^T  \\
        h\Bld  & H\BBd  & A\B^T  \\
        a_l    & A\B    &  -M     }
 \pmat{ \m \dx_l \\
        \m \dx\B \\
         - \dy     }
 = \pmat{ \dz_l \\
           0    \\
           0      }.
\end{equation}
Under the given assumption that $K_l$ is nonsingular, the vectors $\dx_l$,
$\dx\B$ and $\dy$ are uniquely determined by (\ref{eqn-dzzero}) for a fixed
value of $\dz_l$. In addition, once $\dx_l$, $\dx\B$ and $\dy$ are defined,
$\dz\N$ is uniquely determined by (\ref{eqn-dzN}).  It follows that if
$\dz_l=0$, then $\dx_l=0$, $\dx\B=0$, $\dy=0$ and $\dz\N=0$.

It remains to show that if $\dz_l>0$, then either
(\ref{prop-Klnonsing:dx-pos}) or (\ref{prop-Klnonsing:dx-zero}) must hold.
If $K\B$ is singular, then Proposition~\ref{propA-nonsing} of the Appendix
implies that there must exist $u$ and $v$ such that
\[
 \pmat{ H\BB \\
        A\B } u
 = \pmat{ 0 \\
          0 }          \wordss{and}
   \pmat{  A\B^T  \\
          -M       } v
 = \pmat{0 \\ 0 }.
\]
Proposition~\ref{propA-diag} of the Appendix implies that the vector $u$
must also satisfy $h\Bl^T u=0$.
If $u$ is nonzero, then $(0,u,0)$ is a nontrivial null vector for $K_l$,
which contradicts the assumption that $K_l$ is nonsingular.  It follows
that $\tmat{H\BB }{ A\B^T}$ has full row rank and the singularity of $K\B$
must be caused by dependent rows in $\tmat{A\B}{-M}$. The nonsingularity
of $K_l$ implies that $\tmatt{a_l}{A\B}{-M}$ has full row rank and there
must exist a vector $v$ such that $v\T a_l\ne 0$, $v\T A\B=0$ and $v\T
 M=0$. If $v$ is scaled so that $v\T a_l=-\dz_l$, then $(0,0,-v)$ must be a
solution of (\ref{eqn-dzzero}).  It follows that $\dx_l=0$, $v = \dy$,
and $(0,\dy)$ is an eigenvector of $K\B$ associated with a zero eigenvalue.
The nonsingularity of $K_l$ implies that $v$ is unique given the value of
the scalar $\dz_l$, and hence the zero eigenvalue has multiplicity one.

Conversely, $\dx_l=0$ implies that $(\dx\B,\dy)$ is a null vector for $K\B$.
However, if $K\B$ is nonsingular, then the vector is zero, contradicting
(\ref{eqn-KBmain}).  It follows that $K\B$ must be singular.
\end{proof}

\section{A Primal Active-Set Method for Convex QP} \label{sec-primal}
In this section a primal-feasible method for convex \QP{} is formulated.
Each iteration begins and ends with a point $(x,y,z)$ that satisfies the
conditions
\begin{equation}\label{eqn-primalxyz}
\begin{alignedat}{3}
 H x + c - A\T y -z &=   0, & \qquad  x\N + q\N &= 0, & \qquad x\B + q\B &\ge 0, \\
      Ax +  M y - b &=   0, & \qquad  z\B + r\B &= 0,
\end{alignedat}
\end{equation}
for appropriate second-order consistent bases.  The purpose of the
iterations is to drive $(x,y,z)$ to optimality by driving the dual
variables to feasibility (i.e., by driving the negative components of $z\N +
r\N$ to zero).  Methods for finding $\setB$ and $\setN$ at the initial
point are discussed in Section~\ref{sec-primal-dual}.

An iteration consists of a group of one or more consecutive
\emph{subiterations} during which a specific dual variable is made
feasible.  The first subiteration is called the \emph{base}
subiteration. In some cases only the base subiteration is performed, but,
in general, additional \emph{intermediate} subiterations are required.

At the start of the base subiteration, an index $l$ in the nonbasic set
$\setN$ is identified such that $z_l+r_l < 0$.  The idea is to remove the
index $l$ from $\setN$ (i.e., $\setN\gets \setN \setminus \{l\}$) and
attempt to increase the value of $z_l+r_l$ by taking a step along a
primal-feasible direction $(\dx_l$, $\dx\B$, $\dy$, $\dz_l)$.  The removal
of $l$ from $\setN$ implies that $\setB\cup\{l\}\cup\setN = \{ 1$, $2$,
\dots, $n\}$ with $\setB$ second-order consistent. This implies that $K\B$
is nonsingular and the (unique) search direction may be computed as in
(\ref{eqn-KB}) with $\dx_l=1$.

If $\dz_l > 0$, the step $\alphastar = -(z_l+r_l)/\dz_l$ gives
$z_l+\alphastar \dz_l + r_l = 0$.  Otherwise, $\dz_l=0$, and there is no
finite value of $\alpha$ that will drive $z_l + \alpha \dz_l + r_l$ to its
bound, and $\alphastar$ is defined to be $+\infty$.
Proposition~\ref{propA-dirs} of the Appendix implies that the case $\dz_l
=0$ corresponds to the primal objective function being linear and
decreasing along the search direction.

Even if $\dz_l$ is positive, it is not always possible to take the step
$\alphastar$ and remain primal feasible.  A positive step in the direction
$(\dx_l$, $\dx\B$, $\dy$, $\dz_l)$ must increase $x_l$ from its bound, but
may decrease some of the basic variables. This makes it necessary to limit
the step to ensure that the primal variables remain feasible. The largest
step length that maintains primal feasibility is given by
\[
  \alphamax = \min_{ i \st \dx_i < 0} \sgap \frac{x_i + q_i}{-\dx_i}.
\]
If $\alphamax$ is finite, this value gives $x_k + \alphamax \dx_k + q_k =
0$, where $k$ is the index $k = \argmin_{i \st \dx_i < 0}\sgap (x_i +
q_i)/(-\dx_i)$.  The overall step length is then $\alpha =
\min\big(\alphastar,\alphamax\big)$.  An infinite value of $\alpha$
indicates that the primal problem $(\Primal_{q,r})$ is unbounded, or,
equivalently, that the dual problem $(\Dual_{q,r})$ is infeasible. In this
case, the algorithm is terminated.  If the step $\alpha = \alphastar$ is
taken, then $z_l+ \alpha \dz_l+r_l=0$, the subiterations are terminated
with no intermediate subiterations and $\setB \gets \setB \cup \{ l \}$.
Otherwise, $\alpha = \alphamax$, and the basic and nonbasic sets are
updated as $\setB \gets \setB\setminus \{k\}$ and $\setN \gets \setN\cup
\{k\}$ giving a new partition $\setB\cup\{l\}\cup\setN = \{ 1$, $2$,
\dots, $n\}$.  In order to show that the equations associated
with the new partition are well-defined, it is necessary to show that
allowing $z_k$ to move does not give a singular $K_l$.
Proposition~\ref{propA-primalnonsing} of the Appendix shows that the
submatrix $K_l$ associated with the updated $\setB$ and $\setN$ is
nonsingular for the cases $\dz_l>0$ and $\dz_l=0$.

Because the removal of $k$ from $\setB$ does not alter the nonsingularity
of $K_l$, it is possible to add $l$ to $\setB$ and thereby define a unique
solution of the system (\ref{eqn-B-N-lin}). However, if $z_l+r_l<0$, additional
intermediate subiterations are required to drive $z_l + r_l$ to zero. In
each of these subiterations, the search direction is computed by choosing
$\dz_l=1$ in Proposition~\ref{prop-Klnonsing}.  The step length
$\alphastar$ is given by $\alphastar = -(z_l+r_l)/\dz_l$ as in the base
subiteration above, but now $\alphastar$ is always finite because $\dz_l=1$.
Similar to the base subiteration, if no constraint is added, then
$z_l+\alphastar \dz_l+r_l=0$. Otherwise, the index of another blocking
variable $k$ is moved from $\setB$ to
$\setN$. Proposition~\ref{propA-primalnonsing} implies that the updated matrix
$K_l$ is nonsingular at the end of an intermediate subiteration. As a
consequence, the intermediate subiterations may be repeated until $z_l +
r_l$ is driven to zero.

At the end of the base subiteration or after the intermediate subiterations
are completed, it must hold that $z_l + r_l = 0$ and the final $K_l$ is
nonsingular.  This implies that a new iteration may be initiated with the
new basic set $\setB \cup \{l\}$ defining a nonsingular $K\B$.

The primal active-set method is summarized in Algorithm~\ref{alg-primalqp}
below. The convergence properties of Algorithm~\ref{alg-primalqp} are
established in Section~\ref{sec-primal-dual}, which concerns a general
primal algorithm that includes Algorithm~\ref{alg-primalqp} as a special
case.

\begin{algorithm}
\caption{\label{alg-primalqp}\sgap A primal active-set method for convex \QP\@.}
\begin{algorithmic}[0]
  \State Find $(x,y,z)$ satisfying conditions (\ref{eqn-primalxyz})
  for some second-order consistent basis $\setB$;

         \While{$\exists\ l \st z_l+r_l<0$}
\State      $\setN\gets \setN\setminus \{l\}$;
\State      \Call{\algname{primal\_base}}{$\setB$, $\setN$, $l$, $x$, $y$, $z$};
                                     \Comment{returns $\setB$, $\setN$, $x$, $y$, $z$}
            \While{$z_l + r_l < 0$}
\State        \Call{\algname{primal\_intermediate}}{$\setB$, $\setN$, $l$, $x$, $y$, $z$};
                                     \Comment{returns $\setB$, $\setN$, $x$, $y$, $z$}
            \EndWhile
\State      $\setB\gets \setB\cup \{l\}$;
         \EndWhile
  \end{algorithmic}

\begin{algorithmic}[0]
\Function{\algname{primal\_base}}{$\setB$, $\setN$, $l$, $x$, $y$, $z$}
\State         $\dx_l \gets 1$; \quad Solve
               $\pmat{ H\BB & A\B^T \\
                       A\B  & -M      }
                \pmat{ \m \dx\B \\
                        - \dy }
               = -\pmat{ h\Bl \\
                         a_l    }$;
\State         $\dz\N \gets h\Nld \dx\ld + H\BN^T \dx\Bd - A\N^T \dy$;
\State         $\dz_l \gets h\lld \dx\ld + h\Bl^T \dx\Bd - a_l^T \dy$;
               \Comment{$\dz_l \ge 0$}
\State         $\alphastar \gets -(z_l+r_l)/\dz_l$;
               \Comment{$\alphastar \gets+\infty$ if $\dz_l = 0$}
\State         $\disp \alphamax \gets \min_{i \st \dx_i < 0}\mgap(x_i + q_i)/(-\dx_i)$; \agap
               $\disp k \gets \argmin_{i \st \dx_i <   0}\mgap(x_i + q_i)/(-\dx_i)$;
\State         $\alpha \gets \min\big(\alphastar,\alpha\submax \big)$;
               \If{$\alpha = +\infty$}
\State            $\STOP$;     \Comment{$(\Dual_{q,r})$ is infeasible}
               \EndIf
\State         $x_l \gets x_l + \alpha \dx_l$; \agap
               $x\B \gets x\B + \alpha \dx\B$;
\State         $y   \gets y   + \alpha \dy$;   \agap
               $z_l \gets z_l + \alpha \dz_l$; \agap
               $z\N \gets z\N + \alpha \dz\N$;
               \If{$z_l+r_l<0$}
\State            $\setB \gets \setB \setminus \{k\}$; \agap $\setN\gets \setN \cup \{k\}$;
               \EndIf
\State         \Return{$\setB$, $\setN$, $x$, $y$, $z$};
\EndFunction
\end{algorithmic}
\begin{algorithmic}[0]
\Function{\algname{primal\_intermediate}}{$\setB$, $\setN$, $l$, $x$, $y$, $z$}
\State         $\dz_l \gets 1$; \agap Solve
               $\pmat{ h\lld & h\Bl^T & a_l^T \\
                       h\Bld & H\BBd  & A\B^T \\
                       a_l   & A\B    & -M }
                \thnmtx{l}{ \m \dx_l \\
                            \m \dx\B \\
                             - \dy }
                = \pmat{ 1 \\
                         0 \\
                         0   }$;
               \Comment{$\dx_l \ge 0$}
\State         $\dz\N \gets H\Nld \dx\ld + H\BN^T \dx\B - A\N^T \dy$;
\State         $\disp \alphastar    \gets -(z_l+r_l)$;
\State         $\disp \alpha\submax \gets \min_{i \st \dx_i < 0}   \mgap(x_i + q_i)/(-\dx_i)$; \agap
               $\disp             k \gets \argmin_{i \st \dx_i < 0}\mgap(x_i + q_i)/(-\dx_i)$;
\State         $\alpha \gets \min\big(\alphastar,\alpha\submax \big)$;
\State         $x_l \gets x_l + \alpha \dx_l$; \agap $x\B \gets x\B + \alpha \dx\B$;
\State         $y   \gets y   + \alpha \dy$;   \agap $z_l \gets z_l + \alpha \dz_l$;
                                               \agap $z\N \gets z\N + \alpha \dz\N$;
               \If{$z_l+r_l<0$}
\State            $\setB \gets \setB \setminus \{k\}$; \agap $\setN\gets \setN \cup \{k\}$;
               \EndIf
\State         \Return{$\setB$, $\setN$, $x$, $y$, $z$};
\EndFunction
\end{algorithmic}
\end{algorithm}

\section{A Dual Active-Set Method for Convex QP}  \label{sec-dual}
Each iteration of the dual active-set method begins and ends with a point
$(x,y,z)$ that satisfies the conditions
\begin{equation}\label{eqn-dualxyz}
\begin{alignedat}{2}
 H x + c - A\T y -z &=   0, & \qquad  x\N + q\N &= 0,  \\
      Ax +  M y - b &=   0, & \qquad  z\B + r\B &= 0, & \qquad
          z\N + r\N &\ge 0,
\end{alignedat}
\end{equation}
for appropriate second-order consistent bases. For the dual method, the
purpose is to drive the primal variables to feasibility (i.e., by driving
the negative components of $x + q$ to zero).

An iteration begins with a base subiteration in which an index $l$ in the
basic set $\setB$ is identified such that $x_l+q_l < 0$.  The corresponding
dual variable $z_l$ may be increased from its current value $z_l = -r_l$ by
removing the index $l$ from $\setB$, and defining $\setB\gets \setB
\setminus \{l\}$. Once $l$ is removed from $\setB$, it holds that
$\setB\cup\{l\}\cup\setN = \{1$, $2$,\dots, $n\}$.  The resulting matrix
$K_l$ of (\ref{eqn-KB-Kl-defined}) is nonsingular, and the unique direction
$(\dx_l,\dx\B,\dy)$ may be computed with $\dz_l=1$ in
Proposition~\ref{prop-Klnonsing}.

If $\dx_l > 0$, the step $\alphastar = -(x_l+q_l)/\dx_l$ gives $x_l +
\alphastar \dx_l + q_l = 0$.  Otherwise, $\dx_l=0$ and
Proposition~\ref{propA-dirs} of the Appendix implies that the dual
objective function is linear and increasing along $(\dx,\dy,\dz)$. In this
case $\alphastar= +\infty$.  As $x_l + q_l$ is increased towards zero, some
nonbasic dual variables may decrease and the step must be limited
by
$\alphamax = \min_{ i \st \dz_i < 0} \sgap (z_i + r_i)(-\dz_i)$
to maintain feasibility of the nonbasic dual variables. This gives the step
$\alpha = \min\big(\alphastar,\alphamax\big)$. If $\alpha = +\infty$, the
dual problem is unbounded and the iteration is terminated. This is
equivalent to the primal problem $(\Primal_{q,r})$ being infeasible.  If
$\alpha=\alphastar$, then $x_l+ \alpha \dx_l+q_l=0$.  Otherwise, it must
hold that $\alpha = \alphamax$ and $\setN$ and $\setB$ are redefined as
$\setN=\setN\setminus \{k\}$ and $\setB=\setB\cup \{k\}$, where $k$ is the
index $k = \argmin_{ i \st \dz_i < 0}\mgap (z_i + r_i)/(-\dz_i)$.  The
partition at the new point satisfies $\setB\cup\{l\}\cup\setN = \{1$, $2$,
\dots, $n\}$.  Proposition~\ref{propA-dualnonsing} of the Appendix shows
that the new $K\B$ is nonsingular for both of the cases $\dx_l>0$ and
$\dx_l=0$.

If $x_l+q_l<0$ at the new point, then at least one intermediate
subiteration is necessary to drive $x_l + q_l$ to zero.  The nonsingularity
of $K\B$ implies that the search direction may be computed with $\dx_l=1$
in Proposition~\ref{prop-KBnonsing}. As in the base subiteration, the step
length is $\alphastar = -(x_l+q_l)/\dx_l$, but in this case $\alphastar$
can never be infinite because $\dx_l=1$.  If no constraint index is added
to $\setB$, then $x_l+\alpha \dx_l+q_l=0$. Otherwise, the index $k$ of a
blocking variable is moved from $\setN$ to $\setB$.
Proposition~\ref{propA-dualnonsing} of the Appendix implies that the
updated $K\B$ is nonsingular at the end of an intermediate
subiteration. Once $x_l + q_l$ is driven to zero, the index $l$ is moved to
$\setN$ and a new iteration is started.

\smallskip
The dual active-set method is summarized in Algorithm~\ref{alg-dualqp}
below. Its convergence properties are discussed in
Section~\ref{sec-relaxed-dual}.

\begin{algorithm}
\caption{\label{alg-dualqp}\sgap A dual active-set method for convex \QP\@.}
\begin{algorithmic}[0]
\State   Find $(x,y,z)$ satisfying conditions (\ref{eqn-dualxyz}) for some
              second-order consistent basis $\setB$;

         \While{$\exists\ l \st x_l + q_l<0$}
\State      $\setB \gets \setB \setminus \{l\}$;
\State      \Call{\algname{dual\_base}}{$\setB$, $\setN$, $l$, $x$, $y$, $z$};
                                     \Comment{Base subiteration}
            \While{$x_l+q_l<0$}
\State        \Call{\algname{dual\_intermediate}}{$\setB$, $\setN$, $l$, $x$, $y$, $z$};
                                     \Comment{Intermediate subiteration}
            \EndWhile
\State      $\setN\gets \setN\cup \{l\}$;
         \EndWhile
\end{algorithmic}
\begin{algorithmic}[0]
\Function{\algname{dual\_base}}{$\setB$, $\setN$, $l$, $x$, $y$, $z$}
\State         $\dz_l \gets 1$; \agap Solve
               $ \pmat{ h\lld & h\Bl^T & a_l^T \\
                        h\Bl  & H\BB   & A\B^T \\
                        a_l   & A\B    & -M}
                 \pmat{ \m \dx_l \\
                        \m \dx\B \\
                        -  \dy }
               = \pmat{ 1 \\
                        0 \\
                        0   }$;
            \Comment{$\dx_l \ge 0$}
\State         $\dz\N \gets h\Nld \dx\ld + H\BN^T \dx\Bd - A\N^T \dy$;
\State         $\alphastar \gets -(x_l+q_l)/\dx_l$;
               \Comment{$\alphastar \gets+\infty$ if $\dx_l = 0$}
\State         $\disp\alphamax \gets \min_{i \st \dz_i < 0}\mgap (z_i + r_i)/(-\dz_i)$; \agap
               $\disp k \gets \argmin_{ i \st \dz_i < 0}\mgap (z_i + r_i)/(-\dz_i)$;
\State         $\alpha \gets \min\big(\alphastar,\alphamax\big)$;
               \If{$\alpha = +\infty$}
\State             $\STOP$;           \Comment{$(\Primal_{q,r})$ is infeasible}
               \EndIf
\State         $x_l \gets x_l + \alpha \dx_l$; \agap
               $x\B \gets x\B + \alpha \dx\B$;
\State         $y   \gets y + \alpha \dy$; \agap
               $z_l   \gets z_l + \alpha \dz_l$; \agap
               $z\N   \gets z\N + \alpha \dz\N$;
               \If{$x_l+q_l<0$}
\State            $\setB \gets \setB \cup \{k\}$; \agap $\setN\gets \setN
\setminus \{k\}$;
               \EndIf
\State         \Return{$\setB$, $\setN$, $x$, $y$, $z$};
\EndFunction
\end{algorithmic}
\begin{algorithmic}[0]
\Function{\algname{dual\_intermediate}}{$\setB$, $\setN$, $l$, $x$, $y$, $z$}
\State         $\dx_l \gets 1$; \agap Solve
               $\pmat{  H\BB & A\B^T \\
                        A\B  & -M}
                \pmat{ \m \dx\B \\
                        - \dy     }
                = -\pmat{ h\Bl \\
                          a_l }$;
\State         $\dz_l \gets h\lld \dx\ld + h\Bl^T \dx\Bd - a_l^T \dy$;
            \Comment{$\dz_l \ge0$}
\State         $\dz\N \gets h\Nld \dx\ld + H\BN^T \dx\Bd - A\N^T \dy$;
\State         $\alphastar \gets -(x_l + q_l)$;
\State         $\disp \alphamax \gets \min_{ i \st \dz_i < 0}\mgap(z_i + r_i)/(-\dz_i)$; \agap
               $\disp k \gets \argmin_{ i \st \dz_i <   0}\mgap (z_i + r_i)/(-\dz_i);$
\State         $\alpha \gets \min\big(\alphastar,\alphamax\big)$;
\State         $x_l \gets x_l + \alpha \dx_l$; \agap
               $x\B \gets x\B + \alpha \dx\B$;
\State         $y   \gets y + \alpha \dy$;       \agap
               $z_l   \gets z_l + \alpha \dz_l$; \agap
               $z\N   \gets z\N + \alpha \dz\N$;
               \If{$x_l+q_l<0$}
\State            $\setB \gets \setB \cup \{k\}$; \agap $\setN\gets \setN
\setminus \{k\}$;
               \EndIf
\State         \Return{$\setB$, $\setN$, $x$, $y$, $z$};
\EndFunction
\end{algorithmic}
\end{algorithm}

\section{Combining Primal and Dual Active-Set Methods} \label{sec-primal-dual}
The primal active-set method proposed in Section~\ref{sec-primal} may be
used to solve $(\Primal_{q,r})$ for a given initial second-order consistent
basis satisfying the conditions (\ref{eqn-primalxyz}). An appropriate
initial point may be found by solving a conventional phase-1 linear
program.  Alternatively, the dual active-set method of
Section~\ref{sec-dual} may be used in conjunction with an appropriate
phase-1 procedure to solve the quadratic program $(\Primal_{q,r})$ for a
given initial second-order consistent basis satisfying the conditions
(\ref{eqn-dualxyz}).  In this section a method is proposed that provides an
alternative to the conventional phase-1/phase-2 approach. It is shown that
a pair of coupled quadratic programs may be created from the original by
simultaneously shifting the bound constraints.  Any second-order consistent
basis can be made optimal for such a primal-dual pair of shifted
problems. The shifts are then updated using the solution of either the
primal or the dual shifted problem. An obvious application of this approach
is to solve a shifted dual \QP{} to define an initial feasible point for
the primal, or \emph{vice-versa}.  This strategy provides an alternative to
the conventional phase-1/phase-2 approach that utilizes the \QP{} objective
function while finding a feasible point.

\subsection{Finding an initial second-order-consistent basis}\label{sec-SOC-basis}
For the methods described in Section~\ref{sec-initshift} below, it is
possible to define a simple procedure for finding the initial second-order
consistent basis $\setB$ such that the matrix $K\B$ of
(\ref{eqn-submin-KKT}) is nonsingular.  The required basis may be obtained
by finding a symmetric permutation $\Piit$ of the ``full'' KKT matrix $K$
such that
\begin{equation}  \label{eqn-KKTpart}
 \Piit^T K \Piit
  = \Piit^T\pmat{ H & \; A^T  \\
                  A & -M   }\Piit
 = \pmat{ H\BBd  & \;A\B^T & H\BNd \\
          A\B    & -M     & A\N   \\
          H\BN^T & \;A\N^T & H\NNd  },
\end{equation}
where the leading principal block $2\times 2$ submatrix is of the form
(\ref{eqn-submin-KKT}). The full row-rank assumption on $\tmat{A }{
  -M}$ ensures that the permutation (\ref{eqn-KKTpart}) is well
defined, see~\cite[Section 6]{For02}. In practice, the permutation may
be determined using any method for finding a symmetric indefinite
factorization of $K$, see, e.g., \cite{BunP71,Fle76,BunK77}. Such
methods use symmetric interchanges that implicitly form the
nonsingular matrix $K\B$ by deferring singular pivots. In this case,
$K\B$ may be defined as any submatrix of the largest nonsingular
principal submatrix obtained by the factorization.  (There may be
further permutations within $\Piit$ that are not relevant to this
discussion; for further details, see, e.g.,
\cite{FOrM93,For02,DufR83,Duf04}.) The permutation $\Piit$ defines the
initial $\setB$-$\setN$ partition of the columns of $A$, i.e., it defines
an initial second-order consistent basis.

\subsection{Initializing the shifts}\label{sec-initshift}
Given a second-order consistent basis, it is straightforward to create
shifts $(q\up0,r\up0)$ and corresponding $(x,y,z)$ so that $q\up0\ge0$,
$r\up0\ge 0$ and $(x,y,z)$ are optimal for $(\Primal_{q\up0,r\up0})$ and
$(\Dual_{q\up0,r\up0})$. First, choose nonnegative vectors $q\up0\N$ and
$r\up0\B$.  (Obvious choices are $q\up0\N = 0$ and $r\up0\B = 0$.) Define
$z\B = -r\up0\B$, $x\N = -q\up0\N$, and solve the nonsingular \KKT-system
(\ref{eqn-xBy}) to obtain $x\B$ and $y$, and compute $z\N$ from
(\ref{eqn-zN}). Finally, let $q\up0\B\ge \max\{-x\B,0\}$ and $r\up0\N \ge
\max\{-z\N,0\}$. Then, it follows from Proposition~\ref{prop-regQPopt} that
$x$, $y$ and $z$ are optimal for the problems $(\Primal_{q\up0,r\up0})$ and
$(\Dual_{q\up0,r\up0})$, with $q\up0\ge 0$ and $r\up0\ge 0$. If $q\up0$ and
$r\up0$ are zero, then $x$, $y$ and $z$ are optimal for the original
problem.

\subsection{Solving the original problem by removing the shifts}
The original problem may now be solved as a pair of shifted quadratic
programs. Two alternative strategies are proposed.  The first is a ``primal
first'' strategy in which a shifted primal quadratic program is solved,
followed by a dual. The second is an analogous ``dual first'' strategy.

The ``primal-first'' strategy is summarized as follows.
\begin{enumerate}
 \item[(0)] Find $\setB$, $\setN$, $q\up0$, $r\up0$, $x$, $y$, $z$, as described in
   Sections~\ref{sec-SOC-basis} and \ref{sec-initshift}.

 \item[(1)] Set $q\up1 = q\up0$, $r\up1=0$. Solve $(\Primal_{q,0})$ using
  the primal active-set method.

 \item[(2)] Set $q\up2 = 0$, $r\up2=0$. Solve $(\Dual_{0,0})$ using the
  dual active-set method.
\end{enumerate}
In steps (1) and (2), the initial $\setB$--$\setN$ partition and
initial values of $x$, $y$, and $z$ are defined as the final
$\setB$--$\setN$ partition and final values of $x$,
$y$, and $z$ from the preceding step.

The ``dual-first'' strategy is defined in an analogous way.
\begin{enumerate}
 \item[(0)] Find $\setB$, $\setN$, $q\up0$, $r\up0$, $x$, $y$, $z$, as described in
   Section~\ref{sec-SOC-basis} and \ref{sec-initshift}.

\item[(1)] Set $q\up1=0$, $r\up1 = r\up0$. Solve $(\Dual_{0,r})$ using the dual active-set
  method.

\item[(2)] Set $q\up2=0$, $r\up2=0$. Solve $(\Primal_{0,0})$ using the primal active-set
  method.
\end{enumerate}
As in the ``primal-first'' strategy, the initial $\setB$--$\setN$ partition
and initial values of $x$, $y$, and $z$ for steps~(1) and (2), are defined
as the final $\setB$--$\setN$ partition and final values of $x$, $y$, and
$z$ from the preceding step.

(The strategies of solving two consecutive quadratic programs may be
generalized to a sequence of more than two quadratic programs, where
we alternate between primal and dual active-set methods, and eliminate
the shifts in more than two steps.)

In order for these approaches to be well-defined, a simple generalization
of the primal and dual active-set methods of Algorithms~\ref{alg-primalqp}
and \ref{alg-dualqp} is required.

\subsection{Relaxed initial conditions for the primal QP method.}\label{sec-relaxed-primal}
For Algorithm~\ref{alg-primalqp}, the initial values of $\setB$,
$\setN$, $q$, $r$, $x$, $y$, and $z$ must satisfy conditions
(\ref{eqn-primalxyz}). However, the choice of $r = r\up2 = 0$ in
Step~(2) of the dual-first strategy may give some negative components
in the vector $z\B + r\B$.  This possibility may be handled by
defining a simple generalization of Algorithm~\ref{alg-primalqp} that
allows initial points satisfying the conditions
\begin{equation}\label{eqn-primalxyzForShifts}
\begin{alignedat}{3}
  H x + c - A\T y - z &= 0, & \qquad x\N + q\N &=   0, & \qquad x\B + q\B  &\ge 0, \\
       Ax +   M y - b &= 0, & \qquad z\B + r\B &\le 0,
\end{alignedat}
\end{equation}
instead of the conditions (\ref{eqn-primalxyz}). In
Algorithm~\ref{alg-primalqp}, the index $l$ identified at the start of the
primal base subiteration is selected from the set of nonbasic indices such
that $z_j + r_j < 0$.  In the generalized algorithm, the set of eligible
indices for $l$ is extended to include indices associated with negative
values of $z\B + r\B$.  If the index $l$ is deleted from $\setB$, the
associated matrix $K_l$ is nonsingular, and intermediate subiterations are
executed until the updated value satisfies $z_l+r_l=0$. At this point, the
index $l$ is returned $\setB$.  The method is summarized in
Algorithm~\ref{alg-primalqpForShifts}.

\begin{algorithm}
\caption{\label{alg-primalqpForShifts}\sgap A primal active-set method for convex \QP\@.}
\begin{algorithmic}[0]
  \State Find $(x,y,z)$ satisfying conditions
  (\ref{eqn-primalxyzForShifts}) for some second-order consistent basis
  $\setB$;

         \While{$\exists\ l \st z_l+r_l<0$}
            \If{$l \in \setN$}
  \State         $\setN\gets \setN\setminus \{l\}$;
  \State         \Call{\algname{primal\_base}}{$\setB$, $\setN$, $l$, $x$, $y$, $z$};
                                     \Comment{returns $\setB$, $\setN$, $x$, $y$, $z$}
            \Else
  \State         $\setB \gets \setB \setminus \{l\}$;
            \EndIf
            \While{$z_l + r_l < 0$}
  \State        \Call{\algname{primal\_intermediate}}{$\setB$, $\setN$, $l$, $x$, $y$, $z$};
                                     \Comment{returns $\setB$, $\setN$, $x$, $y$, $z$}
            \EndWhile
  \State    $\setB\gets \setB\cup \{l\}$;
         \EndWhile
\end{algorithmic}
\end{algorithm}

This section concludes with a convergence result for the primal method of
Algorithm~\ref{alg-primalqpForShifts}.  In particular, it is shown that the
algorithm is well-defined, and terminates in a finite number of iterations
if $(\Primal_{q,r})$ is \emph{nondegenerate}.
We define nondegeneracy to mean that a nonzero step in the $x$-variables is
taken at each iteration of Algorithm~\ref{alg-primalqpForShifts} that
involves a base subiteration.
A sufficient
condition on $(\Primal_{q,r})$ for this to hold is that the gradients of
the equality constraints and active bound constraints are linearly
independent at each iterate.
See, e.g., Fletcher~\cite{Fle93}
for further discussion of these issues.
As the active-set strategy uses the same criteria for adding and deleting
variables as those used in the simplex method, standard pivot selection
rules used to avoid cycling in linear programming, such as lexicographical
ordering, least-index selection or perturbation may be applied directly to
the method proposed here (see, e.g., Bland~\cite{Bla77},
Charnes~\cite{Cha52}, Dantzig, Orden and Wolfe~\cite{DanOW55}, and
Harris~\cite{Har73}).

\begin{theorem}\label{thm-primalConvergenceWithShifts}
 Given a primal-feasible point $(x_0,y_0,z_0)$ satisfying conditions
 {\rm(\ref{eqn-primalxyzForShifts})} for a second-order consistent basis
 $\setB_0$, then {\rm Algorithm~\ref{alg-primalqpForShifts}} generates a
 sequence of second-order consistent bases $\Seq{\setB_j}_{j>0}$. Moreover,
 if problem {\rm$(\Primal_{q,r})$} is nondegenerate, then {\rm
   Algorithm~\ref{alg-primalqpForShifts}} finds a solution of {\rm
   $(\Primal_{q,r})$} or determines that {\rm$(\Dual_{q,r})$} is infeasible
 in a finite number of iterations.
\end{theorem}
\begin{proof}
Assume that $(x,y,z)$ satisfies the conditions
(\ref{eqn-primalxyzForShifts}) for the second-order consistent basis
$\setB$. Propositions~\ref{prop-KBnonsing} and \ref{prop-Klnonsing} imply
that the \KKT{} matrices associated with subsequent base and intermediate
iterations are nonsingular, in which case each basis is second-order
consistent.  Let $\setBviolated$ denote the index set $\setBviolated=\{ i
\in\setB : z_i + r_i < 0\}$, and let $\rBviolated$ be the vector
$\rBviolated_i=r_i$, $i\not \in \setBviolated$, and $\rBviolated_i=-z_i$,
$i\in \setBviolated$. These definitions imply that $\rBviolated_i=-z_i>
-z_i + z_i + r_i = r_i$, for every $i\in \setBviolated$.  It follows that
$\rBviolated \ge r$, and the feasible region of $(\Dual_{q,r})$ is a subset
of the feasible region of $(\Dual_{q,\rBviolated})$.  In addition, if $r$
is replaced by $\rBviolated$ in (\ref{eqn-primalxyz}), the only difference
is that $z\B+\rBviolated\B=0$, i.e., the initial point for
(\ref{eqn-primalxyzForShifts}) is a stationary point with respect to
$(\Primal_{q,\rBviolated})$.

The first step of the proof is to show that after a finite number of
iterations of Algorithm~\ref{alg-primalqpForShifts}, one of three possible
events must occur: (i) the cardinality of the set $\setBviolated$ is
decreased by at least one; (ii) a solution of problem $(\Primal_{q,r})$ is
found; or (iii) $(\Dual_{q,r})$ is declared infeasible.  The proof will
also establish that if (i) does not occur, then either (ii) or (iii) must
hold after a finite number of iterations.

Assume that (i) never occurs. This implies that the index $l$ selected
in the base subiteration can never be an index in $\setBviolated$
because at the end of such an iteration, it would belong to $\setB$
with $z_l+r_l=0$, contradicting the assumption that the cardinality of
$\setBviolated$ never decreases. For the same reason, it must hold
that $k\not\in \setBviolated$ for every index $k$ selected to be moved
from $\setB$ to $\setN$ in any subiteration, because an index can only
be moved from $\setN$ to $\setB$ by being selected in the base
subiteration.  These arguments imply that $z_i=-\rBviolated_i$, with
$i\in\setBviolated$, throughout the iterations. It follows that the
iterates may be interpreted as being members of a sequence constructed
for solving $(\Primal_{q,\rBviolated})$ with a fixed $\rBviolated$,
where the initial stationary point is given, and each iteration gives
a new stationary point.
The nondegeneracy assumption implies that $\alpha \Deltait x\ne 0$ for
at least one subiteration. For the base subiteration, $\Deltait x_l >
0$, and it follows from Proposition~\ref{prop-Klnonsing} that
$\Deltait x\ne 0$ if and only if $\Deltait x_l>0$ for an intermediate
subiteration.  Therefore, Proposition~\ref{propA-dirs} shows that the
objective value of $(\Primal_{q,\rBviolated})$ is strictly decreasing
for a subiteration where $\alpha \Deltait x\ne 0$. In addition, the
objective value of $(\Primal_{q,\rBviolated})$ is
nonincreasing at each subiteration, so a strict overall improvement of the
objective value of $(\Primal_{q,\rBviolated})$ is obtained at each
iteration.
As there are only a finite number of stationary points,
Algorithm~\ref{alg-primalqpForShifts} either solves
$(\Primal_{q,\rBviolated})$ or concludes that
$(\Dual_{q,\rBviolated})$ is infeasible after a finite number of
iterations. If $(\Primal_{q,\rBviolated})$ is solved, then $z\N+r\N\ge
0$, because $\rBviolated_j = r_j$ for $j\in\setN$.  Hence,
Algorithm~\ref{alg-primalqpForShifts} can not proceed further by
selecting an $l\in\setN$, and the only way to reduce the objective is
to select an $l$ in $\setB$ such that $z_j + r_j < 0$.  Under the
assumption that (i) does not occur, it must hold that no eligible
indices exist and $\setBviolated=\emptyset$. However, in this case
$(\Primal_{q,r})$ has been solved with $\rBviolated=r$, and (ii) must
hold.  If Algorithm~\ref{alg-primalqpForShifts} declares
$(\Dual_{q,\rBviolated})$ to be infeasible, then $(\Dual_{q,r})$ must
also be infeasible because the feasible region of $(\Dual_{q,r})$ is
contained in the feasible region of $(\Dual_{q,\rBviolated})$.  In
this case $(\Dual_{q,r})$ is infeasible and (iii) occurs.

Finally, if (i) occurs, there is an iteration at which the cardinality of
$\setBviolated$ decreases and an index is removed from
$\setBviolated$. There may be more than one such index, but there is at
least one $l$ moved from $\setBviolated$ to $\setB\backslash\setBviolated$,
or one $k$ moved from $\setBviolated$ to $\setN$.  In either case, the
cardinality of $\setBviolated$ is decreased by at least one. After such an
iteration, the argument given above may be repeated for the new set
$\setBviolated$ and new shift $\rBviolated$.  Applying this argument
repeatedly gives the result that the situation (i) can occur only a finite
number of times.

It follows that (ii) or (iii) must occur after a finite number of
iterations, which is the required result.
\end{proof}

\subsection{Relaxed initial conditions for the dual QP method.}\label{sec-relaxed-dual}
Analogous to the primal case, the choice of $q = q\up2 = 0$ in Step~(2) of
the primal-first strategy may give some negative components in the vector
$x\N + q\N$. In this case, the conditions (\ref{eqn-dualxyz}) on the
initial values of $\setB$, $\setN$, $q$, $r$, $x$, $y$, and $z$ are relaxed
so that
\begin{equation}\label{eqn-dualxyzForShifts}
\begin{alignedat}{3}
 H x + c - A\T y - z &= 0, &\qquad x\N + q\N &\le 0, \\
      Ax +   M y - b &= 0, &\qquad z\B + r\B & =  0, &\qquad z\N + r\N &\ge 0.
\end{alignedat}
\end{equation}
Similarly, the set of eligible indices may be extended to include indices
associated with negative values of $x\N + q\N$.  If the index $l$ is from
$\setN$, the associated matrix $K\B$ is nonsingular, and intermediate
subiterations are executed until the updated value satisfies
$x_l+q_l=0$. At this point, the index $l$ is returned $\setN$.  The method
is summarized in Algorithm~\ref{alg-dualqpForShifts}.

\begin{algorithm}
\caption{\label{alg-dualqpForShifts}\sgap A dual active-set method for convex \QP\@.}
\begin{algorithmic}[0]
   \State   Find $(x,y,z)$ satisfying conditions (\ref{eqn-dualxyzForShifts}) for some
                 second-order consistent $\setB$;
            \While{$\exists\ l \st x_l + q_l<0$}
               \If{$l \in \setB$}
   \State         $\setB \gets \setB \setminus \{l\}$;
   \State         \Call{\algname{dual\_base}}{$\setB$, $\setN$, $l$, $x$, $y$, $z$};
                                        \Comment{Base subiteration}
               \Else
   \State         $\setN \gets \setN \setminus \{l\}$;
               \EndIf
               \While{$x_l+q_l<0$}
   \State         \Call{\algname{dual\_intermediate}}{$\setB$, $\setN$, $l$, $x$, $y$, $z$};
                                        \Comment{Intermediate subiteration}
               \EndWhile
   \State      $\setN\gets \setN\cup \{l\}$;
            \EndWhile
\end{algorithmic}
\end{algorithm}

A convergence result analogous to
Theorem~\ref{thm-primalConvergenceWithShifts} holds for the dual algorithm.
In this case, the nondegeneracy assumption concerns the linear independence
of the gradients of the equality constraints and active bounds for
$(\Dual_{q,r})$.

\begin{theorem}\label{thm-dualConvergenceWithShifts}
 Given a dual-feasible point $(x_0,y_0,z_0)$ satisfying conditions
 {\rm(\ref{eqn-dualxyzForShifts})} for a second-order consistent basis
 $\setB_0$, then {\rm Algorithm~\ref{alg-dualqpForShifts}} generates a
 sequence of second-order consistent bases $\Seq{ \setB_j
 }_{j>0}$. Moreover, if problem {\rm$(\Dual_{q,r})$} is non\-degenerate,
 then {\rm Algorithm~\ref{alg-dualqpForShifts}} either solves {\rm$(\Dual_{q,r})$}
 or concludes that {\rm$(\Primal_{q,r})$} is infeasible in a finite number
 of iterations.
\end{theorem}
\begin{proof}
The proof mirrors that of Theorem~\ref{thm-primalConvergenceWithShifts} for
the primal method.
\end{proof}

\section{Practical Issues}\label{sec-general-format}

As stated, the primal quadratic program has lower bound zero on the
$x$-variables. This is for notational convenience. This form may be
generalized in a straightforward manner to a form where the
$x$-variables has both lower and upper bounds on the primal variables,
i.e., $b\subL \le x \le b\U$, where components of $b\subL$ can  be
$-\infty$ and components of $b\U$ can be $+\infty$. Given
primal shifts $q\subL$ and $q\U$, and dual shifts $r\subL$ and $r\U$, we
have the primal-dual pair
\begin{displaymath}
(\Primal_{q,r})\sgap
 \begin{array}{ll}
   \minimize{x,y}  &\m\half x\T H x  + \half y\T  M y + c\T x + (r\subL - r\U)\T x\\
   \subject        &\mbl{- Hx + A\T y  + z\subL - z\U = c}{\m Ax +  M y = b,}
                       \mgap\bgap b\subL - q\subL \le x \le b\U + q\U,
 \end{array}
\]
and
\[
(\Dual_{q,r}) \sgap
 \begin{array}{ll}
   \maximize{x,y,z\subL,z\U}&-\half x\T H x - \half y\T  M y + b\T y + (b\subL - q\subL)\T z\subL - (b\U + q\U)\T z\U\\
   \subject              & \mbl{- Hx + A\T y  + z\subL - z\U = c}{- Hx + A\T y  + z\subL - z\U = c,}
                       \mgap\bgap z\subL \ge -r\subL, \bgap z\U \ge -r\U.
  \end{array}
\end{displaymath}
An infinite bound has neither a shift nor a corresponding dual
variable. For example, if the $j$th components of $b\subL$ and $b\U$ are infinite,
then the corresponding variable $x_j$ is free.
In the procedure given in
Section~\ref{sec-SOC-basis} for finding the first second-order consistent
basis $\setB$, it is assumed that variables with indices not selected for
$\setB$ are initialized at one of their bounds.  As a free variable has no
finite bounds, any index $j$ associated with a free variable should be
selected for $\setB$.  However, this cannot be guaranteed in practice, and
in the next section it is shown that the primal and dual \QP{} methods may
be extended to allow a free variable to be fixed temporarily at some value.

If the \QP{} is defined in the general problem format of
Section~\ref{sec-general-format}, then any free variable not selected for
$\setB$ has no upper or lower bound and must be temporarily fixed at some
value $x_j = \xbar_j$ (say). The treatment of such ``temporary bounds''
involves some additional modifications to the primal and dual methods of
Sections~\ref{sec-relaxed-primal} and \ref{sec-relaxed-dual}.

Each temporary bound $x_j = \xbar_j$ defines an associated dual
variable $z_j$ with initial value $\zbar_j$. As the bound is
temporary, it is treated as an equality constraint, and the desired
value of $z_j$ is zero. Initially, an index $j$ corresponding to a
temporary bound is assigned a primal shift $q_j=0$ and a dual shift
$r_j=-\zbar_j$, making $\xbar_j$ and $\zbar_j$ feasible for the
shifted problem.  In both the primal-first and dual-first approaches,
the idea is to drive the $z_j$-variables associated with temporary
bounds to zero in the primal and leave them unchanged in the dual.

In a primal problem, regardless of whether it is solved before or
after the dual problem, an index $j$ corresponding to a temporary
bound for which $z_j\ne 0$ is considered eligible for selection as $l$
in the base subiteration, i.e., the index can be selected regardless
of the sign of $z_j$. Once selected, $z_j$ is driven to zero and $j$
belongs to $\setB$ after such an iteration. In addition, as $x_j$
has no finite bounds, $j$ will remain in $\setB$ throughout the
iterations. Hence, at termination of a primal problem, any index $j$
corresponding to a temporarily bounded variable must have $z_j=0$. If
the maximum step length at a base subiteration is infinite, the dual
problem is infeasible, as in the case of a regular bound.

In a dual problem, the dual method is modified so that the dual
variables associated with temporary bounds remain fixed throughout the
iterations. At any subiteration, if it holds that $\dz_j\ne 0$ for
some temporary bound, then no step is taken and one such index $j$ is
moved from $\setN$ to $\setB$. Consequently, a move is made only if
$\dz_j= 0$ for every temporary bound $j$.  It follows that the dual
variables for the temporary bounds will remain unaltered throughout
the dual iterations. Note that an index $j$ corresponding to a
temporary bound is moved from $\setN$ to $\setB$ at most once, and is
never moved back because the corresponding $x_j$-variable has no finite
bounds. If the maximum step length at a base subiteration is
infinite, it must hold that $\dz_j=0$ for all temporary bounds $j$,
and the primal problem is infeasible.

The discussion above implies that a pair of primal and dual problems solved
consecutively will terminate with $z_j=0$ for all indices $j$ associated
with temporary bounds. This is because $z_j$ is unchanged in
the dual problem and driven to zero in the primal problem.

\section{Numerical Examples} \label{sec-numerical-results}
This section concerns a particular formulation of the combined primal-dual
method of Section~\ref{sec-primal-dual} in which either a ``primal-first''
or ``dual-first'' strategy is selected based on the initial point.  In
particular, if the point is dual feasible, then the ``dual-first'' strategy
is used, otherwise, the ``primal-first'' strategy is selected.  Some
numerical experiments are presented for a simple \Matlab{} implementation
applied to a set of convex problems from the \CUTEst{} test collection (see
Bongartz, \etal{}\cite{BonCGT95}, and Gould, Orban and Toint~\cite{GouOT03,GouOT15}).

\subsection{The test problems} \label{sec-test-problems}
Each \QP{} problem in the \CUTEst{} test set may be written in the form
\[
  \minimize{x} \mgap \half x\T \Hhat x + c\T x \bgap
  \subject     \mgap \ell \le \pmat{ x \\ \Ahat x } \le u,
\]
where $\ell$ and $u$ are constant vectors of lower and upper bounds, and
$\Ahat$ has dimension $m\times n$.  In this format, a fixed variable or
equality constraint has the same value for its upper and lower bound.  Each
problem was converted to the equivalent form
\begin{equation}\label{eqn-QP-standard}
  \minimize{x,s} \mgap  \half x\T \Hhat x + c\T x \bgap
  \subject       \mgap \Ahat x - s = 0,
                 \bgap \ell \le \pmat{ x \\ s } \le u,
\end{equation}
where $s$ is a vector of slack variables. With this formulation, the \QP{}
problem involves simple upper and lower bounds instead of nonnegativity
constraints.  It follows that the matrix $M$ is zero, but the full row-rank
assumption on the constraint matrix is satisfied because the constraint
matrix $A$ takes the form $\tmat{\Ahat}{-I}$ and has rank $m$.

Numerical results were obtained for a set of 121 convex \QP\@s in standard
interface format (SIF)\@.  The problems were selected based on the
dimension of the constraint matrix $A$ in (\ref{eqn-QP-standard}). In
particular, the test set includes all \QP{} problems for which the smaller
of $m$ and $n$ is of the order of 500 or less.  This gave 121 \QP\@s
ranging in size from \Cute{BQP1VAR} (one variable and one constraint) to
\Cute{LINCONT} (1257 variables and 419 constraints).

\subsection{The implementation}
The combined primal-dual active-set method was implemented in \Matlab{} as
Algorithm~\PDQP\@.  For illustrative purposes, results were obtained for
\PDQP{} and the \QP{} solver \SQOPT{} \cite{GilMS06a}, which is a Fortran
implementation of a conventional two-phase (primal) active-set method for
large-scale \QP\@.  Both \PDQP{} and \SQOPT{} use the method of variable
reduction, which implicitly transforms a \KKT{} system of the form
(\ref{eqn-xBy}) into a block-triangular system.  The general \QP{}
constraints $\Ahat x - s = 0$ are partitioned into the form $B x\s{B} + S
x\s{S} + A\N x\N = 0$, where $B$ is square and nonsingular, with $A\B
=\tmat{B}{S}$ and $x\B = (x\s{B},x\s{S})$.  The vectors $x\s{B}$, $x\s{S}$,
$x\N$ are the associated basic, superbasic, and nonbasic components of
$(x,s)$ (see Gill, Murray and Saunders~\cite{GilMS05}).  If $H$ denotes
the Hessian $\Hhat$ of (\ref{eqn-QP-standard}) augmented by the zero rows and
columns corresponding to the slack variables, then the reduced Hessian $Z\T
H Z$ is defined in terms of the matrix $Z$ such that
\[
    Z  = P \pmat{-B\inv S \\ I \\ 0},
\]
where $P$ permutes the columns of $\tmat{\Ahat}{-I}$ into the order
$\tmatt{B}{S}{A\N}$. The matrix $Z$ is used only as an operator, i.e., it
is not stored explicitly.  Products of the form $Zv$ or $Z\T u$ are
obtained by solving with $B$ or $B^T$.  With these definitions, the
resulting block lower-triangular system has diagonal blocks $Z\T H Z$,
$B$ and $B^T$.

The initial nonsingular $B$ is identified using an LU factorization of
$A^T$.  The resulting $Z$ is used to form $Z\T H Z$, and a partial
Cholesky factorization with interchanges is be used to find an
upper-triangular matrix $R$ that is the factor of the largest nonsingular
leading submatrix of $Z\T H Z$.  If $Z\R$ denotes the columns of $Z$
corresponding to $R$, and $Z$ is partitioned as $Z = \tmat{Z\R}{Z\A}$, then
the index set $\setB$ consisting of the union of the column indices of
$B$ and the indices corresponding to $Z\R$ defines an appropriate initial
second-order consistent basis.

All \SQOPT{} runs were made using the default parameter options.  Both
\PDQP{} and \SQOPT{} are terminated at a point $(x,y,z)$ that satisfies the
optimality conditions of Proposition~\ref{prop-regQPopt} modified to
conform to the constraint format of (\ref{eqn-QP-standard}).  The
feasibility and optimality tolerances are given by $\epsfeas = 10^{-6}$ and
$\epsilon\opt = 10^{-6}$, respectively. For a given $\epsilon\opt$, \PDQP{}
and \SQOPT{} terminate when
\[
  \max_{i\in\setB} \mod{z_i} \le \epsilon\opt\infnorm{y},  \wordss{and}
  \left\{
    \begin{array}{c@{\hspace{10pt}}l}
      z_i \ge  -\epsilon\opt\infnorm{y}  & \mbox{if $x_i \ge - \ell_i$, $i\in\setN$;} \\[2ex]
      z_i \le \m\epsilon\opt\infnorm{y}  & \mbox{if $x_i \le\m u_i$,    $i\in\setN$.}
    \end{array}
  \right.
\]
Both \PDQP{} and \SQOPT{} use the \EXPAND{} anti-cycling procedure of Gill
\etal\ \cite{GilMSW89a} to allow the variables $(x,s)$ to move outside
their bounds by as much as $\epsfeas$.
The \EXPAND{} procedure does not guarantee that cycling will never occur
(see Hall and McKinnon~\cite{HalM96} for an example).  Nevertheless,
in many years of use, the authors have never known \EXPAND{} to cycle on a
practical problem.

\subsection{Numerical results} \label{sec-result-summary}
\PDQP{} and \SQOPT{} were applied to the 121 problems considered in
Section~\ref{sec-test-problems}. A summary of the results is given in
Table~\ref{tab:SQvsPDQP-auto}. The first four columns give the name of the
problem, the number of linear constraints \texttt{m}, the number of
variables \texttt{n}, and the optimal objective value \texttt{Objective}.
The next two columns summarize the \SQOPT{} result for the given problem,
with \texttt{Phs1} and \texttt{Itn} giving the phase-one iterations and
iteration total, respectively.  The last four columns summarize the results
for \PDQP\@. The first column gives the total number of primal and dual
iterations \texttt{Itn}. The second column gives the order in which the
primal and dual algorithms were applied, with \texttt{PD} indicating the
``primal-first'' strategy, and \texttt{DP} the ``dual-first'' strategy. The
final two columns, headed by \texttt{p-Itn}, and \texttt{d-Itn}, give the
iterations required for the primal method and the dual method,
respectively.

Of the 121 problems tested, two (\Cute{LINCONT} and \Cute{NASH}) are known
to be infeasible. This infeasibility was identified correctly by both
\SQOPT{} and \PDQP\@.  In total, \SQOPT{} solved 117 of the remaining 119
problems, but declared (incorrectly) that problems \Cute{RDW2D51U} and
\Cute{RDW2D52U} are unbounded. \PDQP{} solved the same number of problems,
but failed to achieve the required accuracy for the problems
\Cute{RDW2D51B} and \Cute{RDW2D52F}. In these two cases, the final
objective values computed by \PDQP{} were \texttt{1.0947648E-02} and
\texttt{1.0491239E-02} respectively, instead of the optimal values
\texttt{1.0947332e-02} and \texttt{1.0490828e-02}.  (The five
\Cute{RDW2D5*} problems in the test set are known to be difficult to solve,
see Gill and Wong~\cite{GilW15}.)

Figure~\ref{fig:pp} gives a performance profile (in $\log_2$ scale) for the
iterations required by \PDQP{} and \SQOPT\@. (For more details on the use
of performance profiles, see Dolan and Mor\'{e} \cite{DolM02}.) The figure
profiles the total iterations for \PDQP\@, the number of phase-2 iterations
for \SQOPT\@, and the sum of phase-1 and phase-2 iterations for \SQOPT\@.
Some care must be taken when interpreting the results in the profile.
First, the \CUTEst{} test set contains several groups made up of similar
variants of the same problem. In this situation, the profiles can be skewed
by the fact that a method will tend to exhibit similar performance on all
the problems in the group.  For example, \PDQP{} performs significantly
better than \SQOPT{} on all four \Cute{JNLBRNG*} problems, but
significantly worse on all 12 \Cute{LISWET*} problems.

Second, the phase-1 search direction for \SQOPT{} requires the computation
of the vector $-Z Z\T \ghat(x)$, where $\ghat(x)$ is the gradient of the
sum of infeasibilities of the bound constraints at $x$.  This implies that
a phase-1 iteration for \SQOPT{} requires solves with $B$ and $B^T$,
compared to solves with $B$, $B^T$ and $Z\T H Z$ for a phase-2
iteration.  As every iteration for \PDQP{} requires the solution of a
\KKT{} system, if the number of superbasic variables is not small, a
phase-1 iteration of \SQOPT{} requires considerably less work than an
iteration of \PDQP\@. It follows that the total iterations for \PDQP{} and
\SQOPT{} are not entirely comparable. In particular a profile
that would provide an accurate comparison with \PDQP{} lies somewhere
in-between the two \SQOPT{} profiles shown.

Notwithstanding these remarks, the profile indicates that \PDQP{} has
comparable overall performance to a primal method that ignores the
objective function while finding an initial feasible point.  This provides
some preliminary evidence that a combined primal-dual active set method can
be an efficient and reliable alternative to conventional two-phase
active-set methods.  The relative performance of the proposed method is
likely to increase when solving a sequence of related \QP\@s for which the
initial point for one \QP{} is close to being the solution for the next.
In this case, regardless of whether a primal or dual method is being used
to solve the \QP{}, the initial point may start off being primal or dual
feasible, or the number of primal or dual infeasibilities may be small.
This is typically the case for \QP{} subproblems arising in sequential
quadratic programming methods or mixed-integer \QP\@.

Figure~\ref{fig:bar2} provides a bar graph of the so-called ``outperforming
factors'' for iterations, as proposed by Morales~\cite{Mor02}.
On the $x$-axis, each bar corresponds to a particular test problem, with the problems
listed in the order of Table~\ref{tab:SQvsPDQP-auto}.  The $y$-axis
indicates the factor ($\log_2$ scaled) by which one solver  outperformed
the other. A bar in the positive region indicates that \PDQP{} outperformed
\SQOPT.  A negative bar means \SQOPT{} performed better.
A positive (negative) dark grey bar denotes a failure in \SQOPT{} (\PDQP).
Light grey bars denote a zero iteration count for a solver.

{\texttt{\scriptsize
    \begin{longtable}{|lrrr|rr|rcrr|}
      \caption{\label{tab:SQvsPDQP-auto} Results for \texttt{PDQP} and \texttt{SQOPT} on 121  CUTEst QPs.}
      \\\hline
      \multicolumn{1}{|c}{}&%
      \multicolumn{1}{c}{}&%
      \multicolumn{1}{c}{}&%
      \multicolumn{1}{c|}{}&%
      \multicolumn{2}{c|}{SQOPT\hstrt}&%
      \multicolumn{4}{c|}{PDQP}\\
      Name & m & n &
      \multicolumn{1}{c|}{Objective} &%
      \multicolumn{1}{c}{Phs1}&%
      \multicolumn{1}{c|}{Itn}&%
      \multicolumn{1}{c}{Itn}&%
      \multicolumn{1}{c}{Order}&
      \multicolumn{1}{c}{P-Itn}&
      \multicolumn{1}{c|}{D-Itn}\\\hline
      \endfirsthead
      \caption[]{ Results for \texttt{PDQP} and \texttt{SQOPT} on 121 CUTEst QPs. (continued)} \\\hline
      \multicolumn{1}{|c}{}&%
      \multicolumn{1}{c}{}&%
      \multicolumn{1}{c}{}&%
      \multicolumn{1}{c|}{}&%
      \multicolumn{2}{c|}{SQOPT\hstrt}&%
      \multicolumn{4}{c|}{PDQP}\\
      Name & m & n &
      \multicolumn{1}{c|}{Objective}&%
      \multicolumn{1}{c}{Phs1}&%
      \multicolumn{1}{c|}{Itn}&%
      \multicolumn{1}{c}{Itn}&%
      \multicolumn{1}{c}{Order}&
      \multicolumn{1}{c}{P-Itn}&
      \multicolumn{1}{c|}{D-Itn}\\\hline
      \endhead
      \hline
      \endfoot
      \hline
      \multicolumn{10}{|c|}
      {i = infeasible, f = failed\hstrt}\\\hline
      \endlastfoot
  ALLINQP    &    50 &    100 &  -9.1592833E+00 &    0 &   45     &   65     & PD &   63 &    2 \\ %
  AUG2DQP    &   100 &    220 &   1.7797215E+02 &    8 &  116     &  440     & PD &  326 &  114 \\ %
  AUG3D      &    27 &    156 &   8.3333333E-02 &    0 &   45     &   45     & DP &    0 &   45 \\ %
  AVGASA     &    10 &      8 &  -4.6319255E+00 &    5 &    8     &    5     & DP &    0 &    5 \\ %
  AVGASB     &    10 &      8 &  -4.4832193E+00 &    5 &    8     &    7     & DP &    0 &    7 \\ %
  BIGGSB1    &     1 &    100 &   1.5000000E-02 &    0 &  103     &  101     & PD &  101 &    0 \\ %
  BQP1VAR    &     1 &      1 &   0.0000000E+00 &    0 &    1     &    1     & DP &    0 &    1 \\ %
  BQPGABIM   &     1 &     50 &  -3.7903432E-05 &    0 &   36     &    7     & PD &    7 &    0 \\ %
  BQPGASIM   &     1 &     50 &  -5.5198140E-05 &    0 &   40     &    8     & PD &    8 &    0 \\ %
  CHENHARK   &     1 &    100 &  -2.0000000E+00 &    0 &  132     &   32     & DP &    0 &   32 \\ %
  CVXBQP1    &     1 &    100 &   2.2725000E+02 &    0 &  100     &  119     & DP &    2 &  117 \\ %
  CVXQP1     &    50 &    100 &   1.1590718E+04 &    5 &   67     &   91     & DP &    1 &   90 \\ %
  CVXQP2     &    25 &    100 &   8.1209404E+03 &    2 &   82     &   85     & DP &    2 &   83 \\ %
  CVXQP3     &    75 &    100 &   1.1943432E+04 &   17 &   46     &  113     & DP &    2 &  111 \\ %
  DEGENQP    &  1005 &     10 &   0.0000000E+00 &    0 &    6     &   18     & PD &   18 &    0 \\ %
  DTOC3      &    18 &     29 &   2.2459038E+02 &    1 &   10     &   17     & DP &    0 &   17 \\ %
  DUAL1      &     1 &     85 &   3.5012967E-02 &    0 &   88     &   88     & PD &   88 &    0 \\ %
  DUAL2      &     1 &     96 &   3.3733671E-02 &    0 &   99     &   99     & PD &   99 &    0 \\ %
  DUAL3      &     1 &    111 &   1.3575583E-01 &    0 &  106     &  106     & PD &  106 &    0 \\ %
  DUAL4      &     1 &     75 &   7.4609064E-01 &    0 &   61     &   61     & PD &   61 &    0 \\ %
  DUALC1     &   215 &      9 &   6.1552516E+03 &    1 &    9     &    4     & DP &    0 &    4 \\ %
  DUALC2     &   229 &      7 &   3.5513063E+03 &    2 &    4     &    4     & DP &    0 &    4 \\ %
  DUALC5     &   278 &      8 &   4.2723256E+02 &    1 &    7     &    6     & DP &    0 &    6 \\ %
  DUALC8     &   503 &      8 &   1.8309361E+04 &    4 &    6     &    8     & DP &    0 &    8 \\ %
  GENHS28    &     8 &     10 &   9.2717369E-01 &    0 &    3     &    5     & DP &    0 &    5 \\ %
  GMNCASE2   &  1050 &    175 &  -9.9444495E-01 &   18 &   99     &   91     & DP &    0 &   91 \\ %
  GMNCASE3   &  1050 &    175 &   1.5251466E+00 &   31 &  100     &   86     & DP &    0 &   86 \\ %
  GMNCASE4   &   350 &    175 &   5.9468849E+03 &   74 &  171     &  175     & DP &    0 &  175 \\ %
  GOULDQP2   &   199 &    399 &   9.0045697E-06 &    0 &  213     &  419     & DP &    0 &  419 \\ %
  GOULDQP3   &   199 &    399 &   5.6732908E-02 &    0 &  200     &  406     & PD &  205 &  201 \\ %
  GRIDNETA   &   100 &    180 &   9.5242163E+01 &    5 &   35     &  134     & PD &   81 &   53 \\ %
  GRIDNETB   &   100 &    180 &   4.7268237E+01 &    0 &   81     &   97     & DP &    0 &   97 \\ %
  GRIDNETC   &   100 &    180 &   4.8352347E+01 &    6 &   93     &  153     & DP &    0 &  153 \\ %
  HS3        &     1 &      2 &   0.0000000E+00 &    0 &    2     &    1     & PD &    1 &    0 \\ %
  HS3MOD     &     1 &      2 &   1.2325951E-32 &    0 &    2     &    1     & PD &    1 &    0 \\ %
  HS21       &     1 &      2 &  -9.9960000E+01 &    0 &    1     &    0     & PD &    0 &    0 \\ %
  HS28       &     1 &      3 &   1.2325951E-32 &    0 &    2     &    0     & PD &    0 &    0 \\ %
  HS35       &     1 &      3 &   1.1111111E-01 &    0 &    5     &    1     & DP &    0 &    1 \\ %
  HS35I      &     1 &      3 &   1.1111111E-01 &    0 &    5     &    1     & DP &    0 &    1 \\ %
  HS35MOD    &     1 &      3 &   2.5000000E-01 &    0 &    1     &    0     & PD &    0 &    0 \\ %
  HS44       &     6 &      4 &  -1.5000000E+01 &    0 &    2     &    4     & PD &    4 &    0 \\ %
  HS44NEW    &     6 &      4 &  -1.5000000E+01 &    0 &    4     &    9     & PD &    9 &    0 \\ %
  HS51       &     3 &      5 &  -8.8817841E-16 &    0 &    2     &    0     & DP &    0 &    0 \\ %
  HS52       &     3 &      5 &   5.3266475E+00 &    0 &    2     &    1     & DP &    0 &    1 \\ %
  HS53       &     3 &      5 &   4.0930232E+00 &    0 &    2     &    1     & DP &    0 &    1 \\ %
  HS76       &     3 &      4 &  -4.6818181E+00 &    0 &    4     &    4     & DP &    0 &    4 \\ %
  HS76I      &     3 &      4 &  -4.6818181E+00 &    0 &    4     &    4     & DP &    0 &    4 \\ %
  HS118      &    17 &     15 &   6.6482045E+02 &    0 &   21     &   23     & DP &    0 &   23 \\ %
  HS268      &     5 &      5 &   7.2759576E-12 &    0 &    8     &    0     & PD &    0 &    0 \\ %
  HUES-MOD   &     2 &    100 &   3.4829823E+07 &    1 &  103     &    7     & DP &    0 &    7 \\ %
  HUESTIS    &     2 &    100 &   3.4829823E+09 &    1 &  103     &    7     & DP &    0 &    7 \\ %
  JNLBRNG1   &     1 &    529 &  -1.8004556E-01 &    0 &  292     &   82     & PD &   82 &    0 \\ %
  JNLBRNG2   &     1 &    529 &  -4.1023852E+00 &    0 &  252     &   42     & PD &   42 &    0 \\ %
  JNLBRNGA   &     1 &    529 &  -3.0795806E-01 &    0 &  292     &  292     & PD &  292 &    0 \\ %
  JNLBRNGB   &     1 &    529 &  -6.5067871E+00 &    0 &  247     &  247     & PD &  247 &    0 \\ %
  KSIP       &  1001 &     20 &   5.7579792E-01 &    0 & 2847     &   36     & DP &    0 &   36 \\ %
  LINCONT    &   419 &   1257 &   infeasible    &  138 &  138\infs&  304\infs& DP &    0 &  304 \\ %
  LISWET1    &   100 &    106 &   2.6072632E-01 &    0 &   52     &  401     & DP &    0 &  401 \\ %
  LISWET2    &   100 &    106 &   2.5876398E-01 &    0 &   63     &  378     & DP &    0 &  378 \\ %
  LISWET3    &   100 &    106 &   2.5876398E-01 &    0 &   64     &  378     & DP &    0 &  378 \\ %
  LISWET4    &   100 &    106 &   2.5876399E-01 &    0 &   61     &  378     & DP &    0 &  378 \\ %
  LISWET5    &   100 &    106 &   2.5876410E-01 &    0 &   58     &  378     & DP &    0 &  378 \\ %
  LISWET6    &   100 &    106 &   2.5876390E-01 &    0 &   67     &  378     & DP &    0 &  378 \\ %
  LISWET7    &   100 &    106 &   2.5895785E-01 &    0 &   68     &  378     & DP &    0 &  378 \\ %
  LISWET8    &   100 &    106 &   2.5747454E-01 &    0 &   94     &  417     & DP &    0 &  417 \\ %
  LISWET9    &   100 &    103 &   2.1543892E+01 &    0 &   28     &  263     & DP &    0 &  263 \\ %
  LISWET10   &   100 &    106 &   2.5874831E-01 &    0 &   68     &  378     & DP &    0 &  378 \\ %
  LISWET11   &   100 &    106 &   2.5704145E-01 &    0 &   68     &  379     & DP &    0 &  379 \\ %
  LISWET12   &   100 &    106 &   9.1994948E+00 &    0 &   37     &  460     & DP &    0 &  460 \\ %
  LOTSCHD    &     7 &     12 &   2.3984158E+03 &    4 &    8     &   16     & DP &    0 &   16 \\ %
  MOSARQP1   &    10 &    100 &  -1.5420010E+02 &    0 &  102     &   52     & DP &    0 &   52 \\ %
  MOSARQP2   &    10 &    100 &  -2.0651670E+02 &    0 &  100     &   33     & DP &    0 &   33 \\ %
  NASH       &    24 &     72 &   infeasible    &    5 &    5\infs&   24\infs& DP &    0 &   24 \\ %
  OBSTCLAE   &     1 &    529 &   1.6780270E+00 &    0 &  605     &  178     & DP &    0 &  178 \\ %
  OBSTCLAL   &     1 &    529 &   1.6780270E+00 &    0 &  263     &  263     & PD &  263 &    0 \\ %
  OBSTCLBL   &     1 &    529 &   6.5193252E+00 &    0 &  469     &  469     & PD &  469 &    0 \\ %
  OBSTCLBM   &     1 &    529 &   6.5193252E+00 &    0 &  484     &  189     & DP &    0 &  189 \\ %
  OBSTCLBU   &     1 &    529 &   6.5193252E+00 &    0 &  303     &  303     & PD &  303 &    0 \\ %
  OSLBQP     &     1 &      8 &   6.2500000E+00 &    0 &    6     &    0     & PD &    0 &    0 \\ %
  PENTDI     &     1 &    500 &  -7.5000000E-01 &    0 &    2     &    2     & PD &    2 &    0 \\ %
  POWELL20   &   100 &    100 &   5.2703125E+04 &   49 &   52     &   99     & DP &    0 &   99 \\ %
  PRIMAL1    &    85 &    325 &  -3.5012967E-02 &    0 &  217     &   70     & PD &   70 &    0 \\ %
  PRIMAL2    &    96 &    649 &  -3.3733671E-02 &    0 &  407     &   97     & PD &   97 &    0 \\ %
  PRIMAL3    &   111 &    745 &  -1.3575583E-01 &    0 & 1223     &  102     & PD &  102 &    0 \\ %
  PRIMAL4    &    75 &   1489 &  -7.4609064E-01 &    0 & 1264     &   63     & PD &   63 &    0 \\ %
  PRIMALC1   &     9 &    230 &  -6.1552516E+03 &    0 &   18     &    5     & PD &    5 &    0 \\ %
  PRIMALC2   &     7 &    231 &  -3.5513063E+03 &    0 &    3     &    5     & PD &    5 &    0 \\ %
  PRIMALC5   &     8 &    287 &  -4.2723256E+02 &    0 &   10     &    6     & PD &    6 &    0 \\ %
  PRIMALC8   &     8 &    520 &  -1.8309432E+04 &    0 &   30     &    6     & PD &    6 &    0 \\ %
  QPCBLEND   &    74 &     83 &  -7.8425425E-03 &    0 &  111     &  182     & PD &  182 &    0 \\ %
  QPCBOEI1   &   351 &    384 &   1.1503952E+07 &  415 & 1055     &  793     & PD &  395 &  398 \\ %
  QPCBOEI2   &   166 &    143 &   8.1719635E+06 &  142 &  315     &  340     & PD &  163 &  177 \\ %
  QPCSTAIR   &   356 &    467 &   6.2043917E+06 &  210 &  433     &  970     & PD &  645 &  325 \\ %
  QUDLIN     &     1 &    420 &  -8.8290000E+06 &    0 &  419     &  419     & PD &  419 &    0 \\ %
  RDW2D51F   &   225 &    578 &   1.1209939E-03 &   29 &   29     &  217     & DP &    0 &  217 \\ %
  RDW2D51U   &   225 &    578 &   8.3930032E-04 &   14 &   16\fail&  219     & DP &    0 &  219 \\ %
  RDW2D52B   &   225 &    578 &   1.0947648E-02 &  349 &  488     &  316\fail& DP &    0 &  314 \\ 
  RDW2D52F   &   225 &    578 &   1.0491239E-02 &   29 &  191     &  414\fail& DP &    0 &  414 \\ 
  RDW2D52U   &   225 &    578 &   1.0455316E-02 &   15 &  318\fail&  219     & DP &    0 &  219 \\ %
  S268       &     5 &      5 &   7.2759576E-12 &    0 &    8     &    0     & PD &    0 &    0 \\ %
  SIM2BQP    &     1 &      2 &   0.0000000E+00 &    0 &    1     &    1     & PD &    1 &    0 \\ %
  SIMBQP     &     1 &      2 &   6.0185310E-31 &    0 &    2     &    1     & PD &    1 &    0 \\ %
  STCQP1     &    30 &     65 &   4.9452085E+02 &    8 &   53     &   20     & DP &    0 &   20 \\ %
  STCQP2     &   128 &    257 &   1.4294017E+03 &   80 &  215     &   73     & DP &    0 &   73 \\ %
  STEENBRA   &   108 &    432 &   1.6957674E+04 &   14 &   89     &  177     & PD &    2 &  175 \\ %
  TAME       &     1 &      2 &   3.0814879E-33 &    0 &    1     &    1     & PD &    1 &    0 \\ %
  TORSION1   &     1 &    484 &  -4.5608771E-01 &    0 &  256     &  256     & PD &  256 &    0 \\ %
  TORSION2   &     1 &    484 &  -4.5608771E-01 &    0 &  544     &  144     & DP &    0 &  144 \\ %
  TORSION3   &     1 &    484 &  -1.2422498E+00 &    0 &  112     &  112     & PD &  112 &    0 \\ %
  TORSION4   &     1 &    484 &  -1.2422498E+00 &    0 &  689     &  288     & DP &    0 &  288 \\ %
  TORSION5   &     1 &    484 &  -2.8847068E+00 &    0 &   40     &   40     & PD &   40 &    0 \\ %
  TORSION6   &     1 &    484 &  -2.8847068E+00 &    0 &  708     &  360     & DP &    0 &  360 \\ %
  TORSIONA   &     1 &    484 &  -4.1611287E-01 &    0 &  272     &  272     & PD &  272 &    0 \\ %
  TORSIONB   &     1 &    484 &  -4.1611287E-01 &    0 &  529     &  128     & DP &    0 &  128 \\ %
  TORSIONC   &     1 &    484 &  -1.1994864E+00 &    0 &  120     &  120     & PD &  120 &    0 \\ %
  TORSIOND   &     1 &    484 &  -1.1994864E+00 &    0 &  681     &  280     & DP &    0 &  280 \\ %
  TORSIONE   &     1 &    484 &  -2.8405962E+00 &    0 &   40     &   40     & PD &   40 &    0 \\ %
  TORSIONF   &     1 &    484 &  -2.8405962E+00 &    0 &  761     &  360     & DP &    0 &  360 \\ %
  UBH1       &    60 &     99 &   1.1473520E+00 &   11 &   40     &  112     & DP &    0 &  112 \\ %
  YAO        &    20 &     22 &   2.3988296E+00 &    0 &    2     &   20     & DP &    0 &   20 \\ %
  ZECEVIC2   &     2 &      2 &  -4.1250000E+00 &    0 &    4     &    5     & PD &    5 &    0 \\ %
\end{longtable}}}

\begin{figure}[t]
  \begin{center}
    \begin{minipage}{.49\textwidth}
      \includegraphics[width=.97\textwidth]{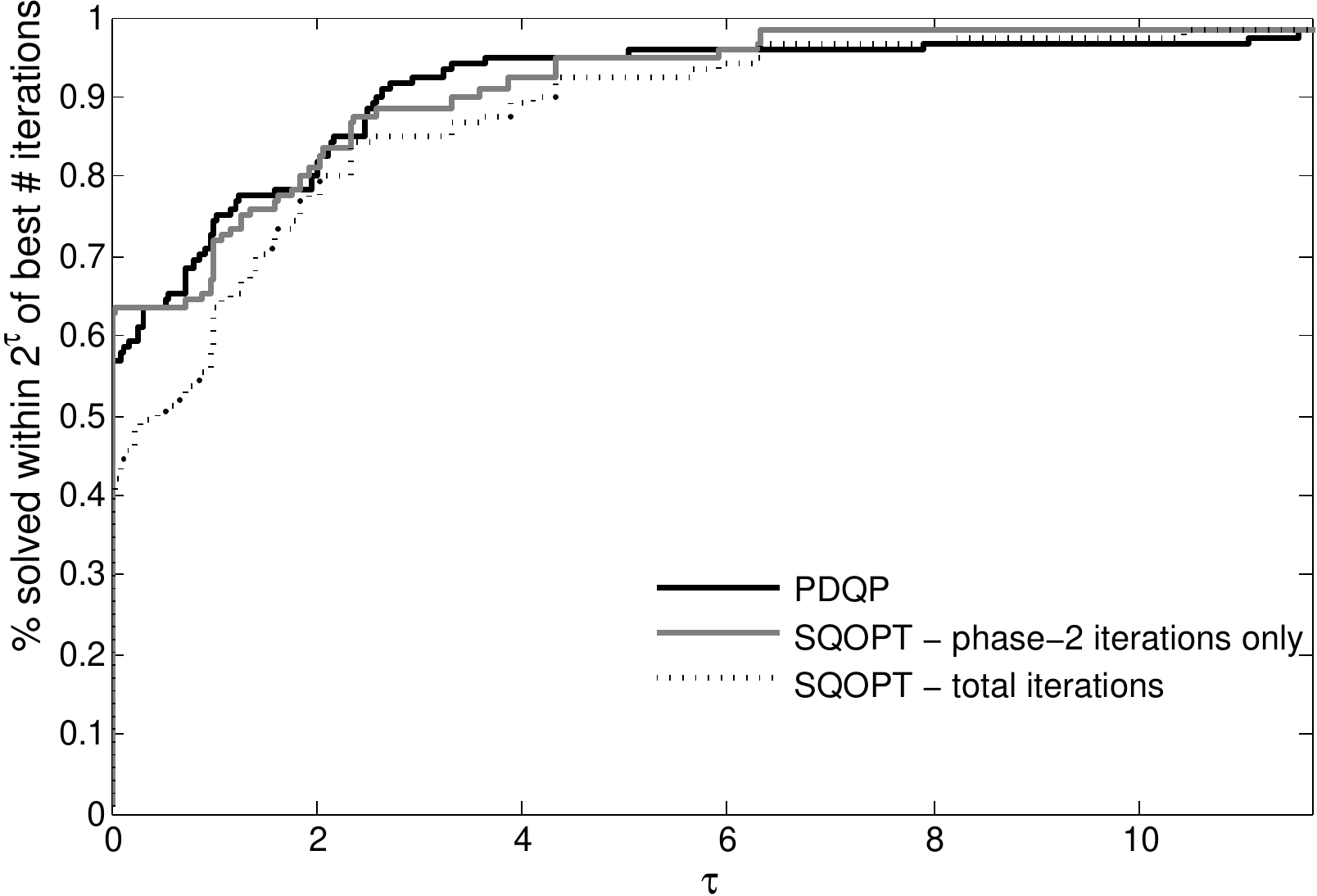}
      \caption{\label{fig:pp} Performance profile of number of iterations
        for \PDQP{} and \SQOPT{} on 121 \texttt{CUTEst} QP problems.}
    \end{minipage}
    \begin{minipage}{.49\textwidth}
      \includegraphics[width=.97\textwidth]{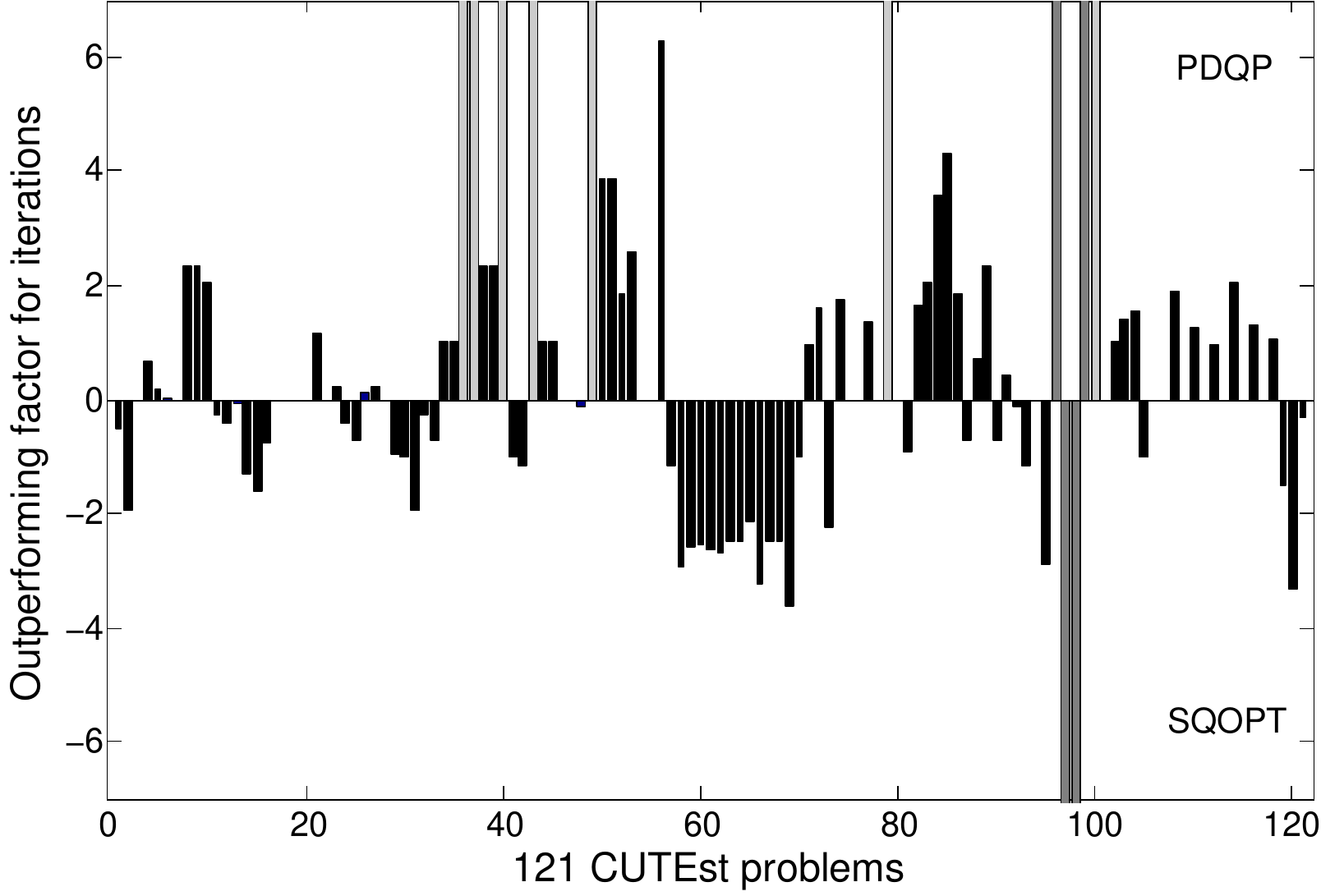}
      \caption{\label{fig:bar2} Outperforming factors for total iterations for each
        of the 121 \texttt{CUTEst} QP problems solved using \PDQP{} and \SQOPT{}.}
    \end{minipage}
  \end{center}
\end{figure}

\section{Summary and Conclusions}

A pair of two-phase active-set methods, one primal and one dual, are
proposed for convex quadratic programming. The methods are derived in terms
of a general framework for solving a convex quadratic program with general
equality constraints and simple lower bounds on the variables.  In each of
the methods, the search directions satisfy a \KKT{} system of equations
formed from Hessian and constraint components associated with an
appropriate column basis.  The composition of the basis is specified by an
active-set strategy that guarantees the nonsingularity of each set of
\KKT{} equations. In addition, a combined primal-dual active set method is
proposed in which a shifted dual \QP{} is solved for a feasible point for
the primal (or \emph{vice versa}), thereby avoiding the need for an initial
feasibility phase that ignores the properties of the objective function.
This approach provides an effective method for finding a dual-feasible
point when the QP is convex but not strictly convex. Preliminary numerical
experiments indicate that this combined primal-dual active set method can
be an efficient and reliable alternative to conventional two-phase
active-set methods.  Future work will focus on the application of the
proposed methods to situations in which a series of related \QP\@s must be
solved, for example, in sequential quadratic programming methods and
methods for mixed-integer nonlinear programming.

\section*{Acknowledgments}
The authors would like to thank two referees for constructive comments that
significantly improved the presentation.

\bibliography{pdqpreferences}

\begin{thebibliography}{10}

\bibitem{AltG99}
A.~Altman and J.~Gondzio.
\newblock Regularized symmetric indefinite systems in interior point methods
  for linear and quadratic optimization.
\newblock {\em Optim. Methods Softw.}, 11/12(1-4):275--302, 1999.

\bibitem{BarB06}
R.~A. Bartlett and L.~T. Biegler.
\newblock {\texttt{QPSchur}}: a dual, active-set, {S}chur-complement method for
  large-scale and structured convex quadratic programming.
\newblock {\em Optim. Eng.}, 7(1):5--32, 2006.

\bibitem{Bea67}
E.~M.~L. Beale.
\newblock An introduction to {B}eale's method of quadratic programming.
\newblock In J.~Abadie, editor, {\em Nonlinear programming}, pages 143--153,
  Amsterdam, the Netherlands, 1967. North Holland.

\bibitem{BeaB78}
E.~M.~L. Beale and R.~Benveniste.
\newblock Quadratic programming.
\newblock In H.~J. Greenberg, editor, {\em Design and Implementation of
  Optimization Software}, pages 249--258. Sijthoff and Noordhoff, The
  Netherlands, 1978.

\bibitem{Bes82}
M.~J. Best.
\newblock An algorithm for the solution of the parametric quadratic programming
  problem.
\newblock CORR 82-14, Department of Combinatorics and Optimization, University
  of Waterloo, Canada, 1982.

\bibitem{Bes96}
M.~J. Best.
\newblock An algorithm for the solution of the parametric quadratic programming
  problem.
\newblock In H.~Fischer, B.~Riedm{\"{u}}ller, and S.~Sch{\"{a}}ffler, editors,
  {\em Applied Mathematics and Parallel Computing: Festschrift for Klaus
  Ritter}, pages 57--76. Physica, Heidelberg, 1996.

\bibitem{Bla77}
R.~G. Bland.
\newblock New finite pivoting rules for the simplex method.
\newblock {\em Math. Oper. Res.}, 2(2):103--107, 1977.

\bibitem{Bol97}
N.~L. Boland.
\newblock A dual-active-set algorithm for positive semi-definite quadratic
  programming.
\newblock {\em Math. Program.}, 78(1, Ser. A):1--27, 1997.

\bibitem{BonCGT95}
I.~Bongartz, A.~R. Conn, N.~I.~M. Gould, and {\relax Ph}.~L. Toint.
\newblock {\texttt{CUTE}}: Constrained and unconstrained testing environment.
\newblock {\em ACM Trans. Math. Software}, 21(1):123--160, 1995.

\bibitem{BunK77}
J.~R. Bunch and L.~Kaufman.
\newblock Some stable methods for calculating inertia and solving symmetric
  linear systems.
\newblock {\em Math. Comput.}, 31:163--179, 1977.

\bibitem{BunK80}
J.~R. Bunch and L.~Kaufman.
\newblock A computational method for the indefinite quadratic programming
  problem.
\newblock {\em Linear Algebra Appl.}, 34:341--370, 1980.

\bibitem{BunP71}
J.~R. Bunch and B.~N. Parlett.
\newblock Direct methods for solving symmetric indefinite systems of linear
  equations.
\newblock {\em SIAM J. Numer. Anal.}, 8:639--655, 1971.

\bibitem{CarG06}
C.~C. Cartis and N.~I.~M. Gould.
\newblock Finding a point in the relative interior of a polyhedron.
\newblock Report RAL-TR-2006-016, Rutherford Appleton Laboratory, Oxon, UK,
  December 2006.

\bibitem{Cha52}
A.~Charnes.
\newblock Optimality and degeneracy in linear programming.
\newblock {\em Econometrica}, 20(2):160--170, 1952.

\bibitem{ChiG15}
A.~Chiche and J.~{\relax Ch}. Gilbert.
\newblock How the augmented {L}agrangian algorithm can deal with an infeasible
  convex quadratic optimization problem.
\newblock {\em Journal of Convex Analysis}, 22(4), 2016.

\bibitem{CurHR14}
F.~E. Curtis, Z.~Han, and D.~P. Robinson.
\newblock {A globally convergent primal-dual active-set framework for
  large-scale convex quadratic optimization}.
\newblock {\em Computational Optimization and Applications}, DOI:
  10.1007/s10589-014-9681-9:1--31, 2014.

\bibitem{DanOW55}
G.~B. Dantzig, A.~Orden, and P.~Wolfe.
\newblock The generalized simplex method for minimizing a linear form under
  linear inequality constraints.
\newblock {\em Pacific J. of Mathematics}, 5:183--195, 1955.

\bibitem{DelG05}
F.~Delbos and J.~{\relax Ch}. Gilbert.
\newblock Global linear convergence of an augmented {L}agrangian algorithm for
  solving convex quadratic optimization problems.
\newblock {\em Journal of Convex Analysis}, 12(1):45--69, 2005.

\bibitem{DolM02}
E.~D. Dolan and J.~J. Mor{\'e}.
\newblock Benchmarking optimization software with performance profiles.
\newblock {\em Math. Program.}, 91(2, Ser. A):201--213, 2002.

\bibitem{Duf04}
I.~S. Duff.
\newblock M{A}57---a code for the solution of sparse symmetric definite and
  indefinite systems.
\newblock {\em ACM Trans. Math. Software}, 30(2):118--144, 2004.

\bibitem{DufR83}
I.~S. Duff and J.~K. Reid.
\newblock The multifrontal solution of indefinite sparse symmetric linear
  equations.
\newblock {\em ACM Trans. Math. Software}, 9:302--325, 1983.

\bibitem{FerBD08}
H.~J. Ferreau, H.~G. Bock, and M.~Diehl.
\newblock An online active set strategy to overcome the limitations of explicit
  {MPC}.
\newblock {\em Internat. J. Robust Nonlinear Control}, 18(8):816--830, 2008.

\bibitem{FerKPBD14}
H.~J. Ferreau, C.~Kirches, A.~Potschka, H.~G. Bock, and M.~Diehl.
\newblock {\texttt{qpOASES}}: a parametric active-set algorithm for quadratic
  programming.
\newblock {\em Mathematical Programming Computation}, pages 1--37, 2014.

\bibitem{Fle71}
R.~Fletcher.
\newblock A general quadratic programming algorithm.
\newblock {\em J. Inst. Math. Applics.}, 7:76--91, 1971.

\bibitem{Fle76}
R.~Fletcher.
\newblock Factorizing symmetric indefinite matrices.
\newblock {\em Linear Algebra Appl.}, 14:257--272, 1976.

\bibitem{Fle93}
R.~Fletcher.
\newblock Resolving degeneracy in quadratic programming.
\newblock {\em Ann. Oper. Res.}, 46/47(2):307--334, 1993.

\bibitem{Fle00}
R.~Fletcher.
\newblock Stable reduced {H}essian updates for indefinite quadratic
  programming.
\newblock {\em Math. Program.}, 87(2, Ser. B):251--264, 2000.
\newblock Studies in algorithmic optimization.

\bibitem{For02}
A.~Forsgren.
\newblock Inertia-controlling factorizations for optimization algorithms.
\newblock {\em Appl. Num. Math.}, 43:91--107, 2002.

\bibitem{FOrM93}
A.~Forsgren and W.~Murray.
\newblock Newton methods for large-scale linear equality-constrained
  minimization.
\newblock {\em SIAM J. Matrix Anal. Appl.}, 14:560--587, 1993.

\bibitem{GilL15}
J.~{\relax Ch}. Gilbert and {\'E}.~Joannopoulos.
\newblock {\texttt{OQLA/QPALM}} -- convex quadratic optimization solvers using
  the augmented {L}agrangian approach, with an appropriate behavior on
  infeasible or unbounded problems.
\newblock Technical report, INRIA, BP 105, 78153 Le Chesnay, France, 2015.

\bibitem{GilGMSW84}
P.~E. Gill, N.~I.~M. Gould, W.~Murray, M.~A. Saunders, and M.~H. Wright.
\newblock A weighted {Gram-Schmidt} method for convex quadratic programming.
\newblock {\em Math. Program.}, 30:176--195, 1984.

\bibitem{GilM78a}
P.~E. Gill and W.~Murray.
\newblock Numerically stable methods for quadratic programming.
\newblock {\em Math. Program.}, 14:349--372, 1978.

\bibitem{GilMS95}
P.~E. Gill, W.~Murray, and M.~A. Saunders.
\newblock User's guide for {{\texttt{QPOPT}} 1.0}: a {Fortran} package for
  quadratic programming.
\newblock Report SOL 95-4, Department of Operations Research, Stanford
  University, Stanford, CA, 1995.

\bibitem{GilMS05}
P.~E. Gill, W.~Murray, and M.~A. Saunders.
\newblock {\texttt{SNOPT}}: An {SQP} algorithm for large-scale constrained
  optimization.
\newblock {\em SIAM Rev.}, 47:99--131, 2005.

\bibitem{GilMS06a}
P.~E. Gill, W.~Murray, and M.~A. Saunders.
\newblock User's guide for {{\texttt{SQOPT}} Version 7}: Software for
  large-scale linear and quadratic programming.
\newblock Numerical Analysis Report 06-1, Department of Mathematics, University
  of California, San Diego, La Jolla, CA, 2006.

\bibitem{GilMSW89a}
P.~E. Gill, W.~Murray, M.~A. Saunders, and M.~H. Wright.
\newblock A practical anti-cycling procedure for linearly constrained
  optimization.
\newblock {\em Math. Program.}, 45:437--474, 1989.

\bibitem{GilMSW90}
P.~E. Gill, W.~Murray, M.~A. Saunders, and M.~H. Wright.
\newblock A {Schur}-complement method for sparse quadratic programming.
\newblock In M.~G. Cox and S.~J. Hammarling, editors, {\em Reliable Numerical
  Computation}, pages 113--138. Oxford University Press, 1990.

\bibitem{GilMSW91}
P.~E. Gill, W.~Murray, M.~A. Saunders, and M.~H. Wright.
\newblock Inertia-controlling methods for general quadratic programming.
\newblock {\em SIAM Rev.}, 33(1):1--36, 1991.

\bibitem{GilR13}
P.~E. Gill and D.~P. Robinson.
\newblock A globally convergent stabilized {SQP} method.
\newblock {\em SIAM J. Optim.}, 23(4):1983--2010, 2013.

\bibitem{GilW15}
P.~E. Gill and E.~Wong.
\newblock Methods for convex and general quadratic programming.
\newblock {\em Math. Prog. Comp.}, 7(1):71--112, 2015.

\bibitem{GolI83}
D.~Goldfarb and A.~Idnani.
\newblock A numerically stable dual method for solving strictly convex
  quadratic programs.
\newblock {\em Math. Programming}, 27(1):1--33, 1983.

\bibitem{Gou91}
N.~I.~M. Gould.
\newblock An algorithm for large-scale quadratic programming.
\newblock {\em IMA J. Numer. Anal.}, 11(3):299--324, 1991.

\bibitem{GouOT03}
N.~I.~M. Gould, D.~Orban, and {\relax Ph}.~L. Toint.
\newblock {\texttt{CUTEr}} and {\texttt{sifdec}}: A constrained and
  unconstrained testing environment, revisited.
\newblock {\em ACM Trans. Math. Software}, 29(4):373--394, 2003.

\bibitem{GouOT03b}
N.~I.~M. Gould, D.~Orban, and {\relax Ph}.~L. Toint.
\newblock {\texttt{GALAHAD}}, a library of thread-safe {Fortran} 90 packages
  for large-scale nonlinear optimization.
\newblock {\em {ACM} Trans. Math. Software}, 29(4):353--372, Dec 2003.

\bibitem{GouOT15}
N.~I.~M. Gould, D.~Orban, and {\relax Ph}.~L. Toint.
\newblock {\texttt{CUTEst}}: a constrained and unconstrained testing
  environment with safe threads for mathematical optimization.
\newblock {\em Computational Optimization and Applications}, 60(3):545--557,
  2015.

\bibitem{GouT02}
N.~I.~M. Gould and {\relax Ph}.~L. Toint.
\newblock An iterative working-set method for large-scale nonconvex quadratic
  programming.
\newblock {\em Appl. Numer. Math.}, 43(1-2):109--128, 2002.
\newblock 19th Dundee Biennial Conference on Numerical Analysis (2001).

\bibitem{GouT02b}
N.~I.~M. Gould and {\relax Ph}.~L. Toint.
\newblock Numerical methods for large-scale non-convex quadratic programming.
\newblock In {\em Trends in industrial and applied mathematics ({A}mritsar,
  2001)}, volume~72 of {\em Appl. Optim.}, pages 149--179. Kluwer Acad. Publ.,
  Dordrecht, 2002.

\bibitem{HalM96}
J.~A.~J. Hall and K.~I.~M. McKinnon.
\newblock The simplest examples where the simplex method cycles and conditions
  where {{\texttt{EXPAND}}} fails to prevent cycling.
\newblock Technical Report MS 96-010, Department of Mathematics and Statistics,
  University of Edinburgh, 1996.

\bibitem{Har73}
P.~M.~J. Harris.
\newblock Pivot selection methods of the {Devex LP} code.
\newblock {\em Math. Program.}, 5:1--28, 1973.
\newblock Reprinted in {\emph{Math. Prog. Study}}, 4, 30--57, 1975.

\bibitem{Hoy86}
S.~C. Hoyle.
\newblock {\em A Single-Phase Method for Quadratic Programming}.
\newblock PhD thesis, Report SOL 86-9, Department of Operations Research,
  Stanford University, Stanford, CA, 1986.

\bibitem{Huy08}
H.~M. Huynh.
\newblock {\em A Large-Scale Quadratic Programming Solver Based on Block-LU
  Updates of the {KKT} System}.
\newblock PhD thesis, Program in Scientific Computing and Computational
  Mathematics, Stanford University, Stanford, CA, 2008.

\bibitem{Mae10}
C.~M. Maes.
\newblock {\em A Regularized Active-Set Method for Sparse Convex Quadratic
  Programming}.
\newblock PhD thesis, Institute for Computational and Mathematical Engineering,
  Stanford University, Stanford, CA, August 2010.

\bibitem{Mor02}
J.~L. Morales.
\newblock A numerical study of limited memory {BFGS} methods.
\newblock {\em Appl. Math. Lett.}, 15(4):481--487, 2002.

\bibitem{PotKBS10}
A.~Potschka, C.~Kirches, H.~Bock, and J.~Schl\"oder.
\newblock Reliable solution of convex quadratic programs with parametric active
  set methods.
\newblock Technical report, Interdisciplinary Center for Scientific Computing,
  Heidelberg University, Im Neuenheimer Feld 368, 69120 Heidelberg, Germany,
  November 2010.

\bibitem{Pow85}
M.~J.~D. Powell.
\newblock On the quadratic programming algorithm of {G}oldfarb and {I}dnani.
\newblock {\em Math. Programming Stud.}, 25:46--61, 1985.

\bibitem{Rit67}
K.~Ritter.
\newblock A method for solving nonlinear maximum problems depending on
  parameters.
\newblock {\em Naval Research Logistics Quarterly}, 14:147--162, 1967.

\bibitem{Rit81}
K.~Ritter.
\newblock On parametric linear and quadratic programming.
\newblock MRC Technical report 2197, University of Wisconsin at Madison,
  Wisconsin, USA, 1981.

\bibitem{Sto86}
J.~Stoer.
\newblock On the realization of the quadratic programming algorithm of
  {G}oldfarb and {I}dnani.
\newblock In {\em Vistas in applied mathematics}, Transl. Ser. Math. Engrg.,
  pages 167--180. Optimization Software, New York, 1986.

\bibitem{Won11}
E.~Wong.
\newblock {\em Active-Set Methods for Quadratic Programming}.
\newblock PhD thesis, Department of Mathematics, University of California San
  Diego, La Jolla, CA, 2011.

\bibitem{SWri98}
S.~J. Wright.
\newblock Superlinear convergence of a stabilized {SQP} method to a degenerate
  solution.
\newblock {\em Comput. Optim. Appl.}, 11(3):253--275, 1998.

\end{thebibliography}
\bibliographystyle{myplain}


\appendix

\section{Appendix}
The appendix concerns some basic results used in previous sections.  The
first result shows that the nonsingularity of a \KKT{} matrix may be
established by checking that the two row blocks $\tmat{ H }{ A^T }$ and
$\tmat{ A }{ -M }$ have full row rank.

\begin{proposition}\label{propA-nonsing}
Assume that $H$ and $M$ are symmetric, positive semidefinite matrices. The
vectors $u$ and $v$ satisfy
\begin{equation}\label{eqn-KKT1}
 \pmat{ H  & A^T \\
        A  & -M   }
 \pmat{\m u \\
        - v }
 = \pmat{ 0 \\
          0 }
\end{equation}
if and only if
\begin{equation}\label{eqn-KKT2}
 \pmat{ H \\
        A   } u
  = \pmat{ 0 \\
           0   }  \words{and}
    \pmat{ A^T \\
          -M  } v
  = \pmat{ 0   \\
           0     }.
\end{equation}
\end{proposition}

\begin{proof}
If (\ref{eqn-KKT2}) holds, then (\ref{eqn-KKT1}) holds, which establishes
the ``if'' direction. Now assume that $u$ and $v$ are vectors such that
(\ref{eqn-KKT1}) holds. Then,
\[
 u\T H u - u\T A\T v = 0, \wordss{and}
 v\T A u + v\T  M  v = 0.
\]
Adding these equations gives the identity $u\T H u + v\T  M v = 0$.  But
then, the symmetry and semidefiniteness of $H$ and $M$ imply $u\T H u=0$
and $v\T  M v = 0$. This can hold only if $H u = 0$ and $M v=0$. If $Hu=0$
and $M v=0$, (\ref{eqn-KKT1}) gives $A\T v=0$ and $Au=0$, which implies
that (\ref{eqn-KKT2}) holds, which completes the proof.
\end{proof}

The next result shows that when checking a subset of the columns of a
symmetric positive semidefinite matrix for linear dependence, it is only
the diagonal block that is of importance. The off-diagonal block may be
ignored.

\begin{proposition}\label{propA-diag}
Let $H$ be a symmetric, positive semidefinite matrix partitioned as
\[
 H = \pmat{ H_{11}    & H_{12}     \\[1pt]
            H_{12}^T  & H_{22}\drop  }.
\]
Then,
\[
 \pmat{ H_{11}  \\[1pt]
        H_{12}^T  } u
 = \pmat{ 0 \\
          0  } \words{if and only if}
 H_{11} u = 0.
\]
\end{proposition}
\begin{proof}
If $H$ is positive semidefinite, then $H_{11}$ is positive semidefinite,
and it holds that
\[
   \pmat{ 0 \\
          0 }
 = \pmat{ H_{11}  \\[1pt]
          H_{12}^T  } u
 = \pmat{ H_{11}    & H_{12}     \\[1pt]
          H_{12}^T  & H_{22}\drop  }
   \pmat{ u \\
          0 }
\]
if and only if
\[
0 =  \pmat{ u^T      & 0 }
     \pmat{ H_{11}   & H_{12}     \\[1pt]
            H_{12}^T & H_{22}\drop  }
     \pmat{ u \\[1pt]
            0 }
  = u\T H_{11} u
\]
if and only if $H_{11} u = 0$, as required.
\end{proof}

In the following propositions, the distinct integers $k$ and $l$, together with
integers from the index sets $\setB$ and $\setN$ define a partition of
$\setI = \{1$, $2$, \dots, $n\}$, i.e., $\setI = \setB \cup \{k\} \cup
\{l\} \cup \setN$.  If $w$ is any $n$-vector, the $n\B$-vector $w\B$ and
$w\N$-vector $w\N$ denote the vectors of components of $w$ associated with
$\setB$ and $\setN$.  For the symmetric Hessian $H$, the matrices $H\BB$
and $H\NN$ denote the subset of rows and columns of $H$ associated with the
sets $\setB$ and $\setN$ respectively. The unsymmetric matrix of components
$h_{ij}$ with $i\in\setB$ and $j\in \setN$ will be denoted by $H\BN$.
Similarly, $A\B$ and $A\N$ denote the matrices of columns associated with
$\setB$ and $\setN$.

The next result concerns the row rank of the $\tmat{A }{ -M}$ block of the
\KKT{} matrix.

\begin{proposition}\label{propA-AMfullrowrank}
If the matrix $\tmattt{ a_l }{ a_k }{ A\B }{ -M }$ has full row rank, and there
exist $\dx_l$, $\dx_k$, $\dx\B$, and $\dy$ such that $a_l \dx_l + a_k \dx_k
+ A\B \dx\B +  M \dy=0$ with $\dx_k\ne 0$, then $\tmatt{ a_l }{ A\B }{ -M }$
has full row rank.
\end{proposition}

\begin{proof}
It must be established that $u\T \pmat{ a_l & A\B & -M }=0$ implies that
$u=0$. For a given $u$, let $\gamma = -u\T a_k$, so that
\[
 \pmat{ u^T & \gamma }
 \pmat{ a_l & a_k & A\B    & -M \\
            & 1                    }
  = \pmat{ 0 & 0 & 0 & 0 }.
\]
Then,
\[
 0
 = \pmat{ u^T & \gamma }
   \pmat{ a_l & a_k & A\B    & -M  \\
              & 1                    }
   \pmat{ \m \dx_l \\
          \m \dx_k \\
          \m \dx\B \\
           - \dy }
 = \gamma \,\dx_k.
\]
As $\dx_k\ne0$, it must hold that $\gamma = 0$, in which case
\[
  u^T \mattt{ a_l }{ a_k }{ A\B }{ -M  } = 0.
\]
As $\tmattt{ a_l }{ a_k }{ A\B }{ -M }$ has full row rank by assumption, it
follows that $u=0$ and $\tmatt{ a_l }{ A\B }{ -M }$ must have full row rank.
\end{proof}

An  analogous result holds concerning the $\tmat{ H }{ A^T }$ block of
the \KKT{} matrix.

\begin{proposition} \label{propA-HATfullrowrank}
If $\tmat{ H\BBd }{ A\B^T }$ has full row rank, and there exist
quantities $\dx\N$, $\dx\B$, $\dy$, and $\dz_k$ such that
\begin{equation}\label{eqn-HATfullrowrank}
  \pmat{  h\Nk^T & h\Bk^T & a_k^T & 1 \\[1pt]
          h\BNd  & H\BBd  & A\B^T      }
  \pmat{ \m\dx\N \\
         \m\dx\B \\
          -\dy   \\
          -\dz_k   }
 = \pmat{ 0 \\
          0 },
\end{equation}
with $\dz_k \ne 0$, then the matrix
\[
  \pmat{  h\kkd  & h\Bk^T & a_k^T \\[1pt]
          h\Bkd  & H\BBd  & A\B^T }
\]
has full row rank.
\end{proposition}

\begin{proof}
Let $\tmat{ \mu }{ v^T }$ be any vector such that
\[
\mat{ \mu }{ v^T }
  \pmat{  h\Nk^T  & h\Bk^T & a_k^T \\[1pt]
          h\BNd   & H\BBd  & A\B^T }
= \pmat{0 & 0 & 0 }.
\]
The assumed identity (\ref{eqn-HATfullrowrank}) gives
\[
0 = \mat{ \mu }{ v^T }
    \pmat{  h\Nk^T  & h\Bk^T & a_k^T \\[1pt]
            h\BNd   & H\BBd  & A\B^T }
    \pmat{ \m\dx\N \\
           \m\dx\B \\
            -\dy   }
  = \mu \,\dz_k.
\]
As $\dz_k\ne 0$ by assumption, it must hold that $\mu=0$. The full row rank
of $\tmat{ H\BBd }{ A\B^T }$ then gives $v = 0$ and
\[
  \pmat{  h\Nk^T   & h\Bk^T & a_k^T \\[1pt]
          h\BNd    & H\BBd  & A\B^T }
\]
must have full row rank. Proposition~\ref{propA-nonsing} implies that this
is equivalent to
\[
  \pmat{  h\kkd & h\Bk^T & a_k^T \\[1pt]
          h\Bkd & H\BBd  & A\B^T }
\]
having full row rank.
\end{proof}

The next proposition concerns the primal subiterations when the constraint index
$k$ is moved from $\setB$ to $\setN$. In particular, it is shown that the
$K_l$ matrix is nonsingular after a subiteration.

\begin{proposition}\label{propA-primalnonsing}
Assume that $(\dx_l$, $\dx_k$, $\dx\B$, $-\dy$, $-\dz_l)$ is the unique solution of
the equations
\begin{equation}\label{eqn-primaldir}
  \pmat{ h\lld & h\kld & h\Bl^T & a_l^T & \m 1 \\[1pt]
         h\kld & h\kkd & h\Bk^T & a_k^T &      \\[1pt]
         h\Bld & h\Bkd & H\BBd  & A\B^T &      \\[1pt]
         a_l   & a_k   & A\B    & -M    &      \\[1pt]
         1     &       &        &       &   -1   }
  \pmat{ \m\dx_l \\[1pt]
         \m\dx_k \\[1pt]
         \m\dx\B \\[1pt]
        -\dy     \\[1pt]
        -\dz_l }
= \pmat{ 0 \\[1pt]
         0 \\[1pt]
         0 \\[1pt]
         0 \\[1pt]
         1   },
\end{equation}
and that $\dx_k\ne 0$. Then, the matrices $K_l$ and $K_k$ are nonsingular,
where
\[
 K_l = \pmat{ h\lld  & h\Bl^T & a_l^T \\[1pt]
              h\Bld  & H\BBd  & A\B^T \\[1pt]
              a_l    & A\B    &  -M     } \wordss{and}
 K_k = \pmat{ h\kkd  & h\Bk^T & a_k^T \\[1pt]
              h\Bkd  & H\BBd  & A\B^T \\[1pt]
              a_k    & A\B    &  -M     }.
\]
\end{proposition}

\begin{proof}
By assumption, the equations (\ref{eqn-primaldir}) have a unique solution
with $\dx_k\ne 0$.  This implies that there is no solution of the
overdetermined equations
\begin{equation}\label{eqn-baseprimal}
   \pmat{ h\lld  & h\kld & h\Bl^T & a_l^T &\m 1 \\[1pt]
          h\kld  & h\kkd & h\Bk^T & a_k^T &     \\[1pt]
          h\Bld  & h\Bkd & H\BBd  & A\B^T &     \\[1pt]
          a_l    & a_k   & A\B    & -M    &     \\[1pt]
          1      &       &        &       & -1  \\[1pt]
                 & 1     &        &       &       }
   \pmat{\m\dx_l \\[1pt]
         \m\dx_k \\[1pt]
         \m\dx\B \\[1pt]
         - \dy   \\[1pt]
         - \dz_l }
 = \pmat{ 0 \\[1pt]
          0 \\[1pt]
          0 \\[1pt]
          0 \\[1pt]
          1 \\[1pt]
          0   }.
\end{equation}
Given an arbitrary matrix $D$ and nonzero vector $f$, the fundamental
theorem of linear algebra implies that if $D w = f$ has no solution, then
there exists a vector $v$ such that $v\T f \ne 0$.  The application of this
result to (\ref{eqn-baseprimal}) implies the existence of a nontrivial
vector $(\dxtilde_l$, $\dxtilde_k$, $\dxtilde\B$, $-\dytilde$,
$-\dztilde_l$, $-\dztilde_k)$ such that
\begin{equation}\label{eqn-baseprimal-alt}
  \pmat{ h\lld & h\kld & h\Bl^T & a_l^T & \m 1  &              \\[1pt]
         h\kld & h\kkd & h\Bk^T & a_k^T &       &\mbc{a_l^T}{1}\\[1pt]
         h\Bld & h\Bkd & H\BBd  & A\B^T &       &              \\[1pt]
         a_l   &  a_k  & A\B    &  -M   &       &              \\[1pt]
         1     &       &        &       &   -1  &     }
  \pmat{ \m\dxtilde_l \\
         \m\dxtilde_k \\
         \m\dxtilde\B \\
         - \dytilde   \\
         - \dztilde_l \\
         - \dztilde_k   }
= \pmat{ 0 \\[1pt]
         0 \\[1pt]
         0 \\[1pt]
         0 \\[1pt]
         0   },
\end{equation}
with $\dztilde_l\ne 0$. The last equation of (\ref{eqn-baseprimal-alt})
gives $\dxtilde_l + \dztilde_l=0$, in which case $\dxtilde_l \dztilde_l =
-\dztilde_l^2 <0$ because $\dztilde_l\ne0$. Any solution
of (\ref{eqn-baseprimal-alt}) may be viewed as a solution of the equations
$H \dxtilde - A\T \dytilde - \dztilde = 0$, $A \dxtilde +  M \dytilde = 0$,
$\dztilde\B=0$, and $\dxtilde_i = 0$ for $i\in \{ 1$, $2$, \dots, $n\}
\setminus \{ l \} \setminus \{ k \}$. An argument similar to that used to
establish Proposition~\ref{prop-dotproduct} gives
\[
  \dxtilde_l \dztilde_l + \dxtilde_k \dztilde_k \ge 0,
\]
which implies that $\dxtilde_k \dztilde_k > 0$, with $\dxtilde_k\ne0$ and
$\dztilde_k\ne 0$.

As the search direction is unique, it follows from (\ref{eqn-primaldir})
that $\tmattt{h\Bld }{ H\Bkd }{ H\BBd }{ A\B^T }$ has full row rank, and
Proposition~\ref{propA-diag} implies that $\tmat{H\BBd }{ A\B^T }$ has full
row rank. Hence, as $\dztilde_l\ne0$, it follows from
(\ref{eqn-baseprimal-alt}) and Proposition~\ref{propA-HATfullrowrank} that
the matrix
\[
  \pmat{  h\lld & h\kld & h\Bl^T & a_l^T \\[1pt]
          h\Bld & h\Bkd & H\BBd  & A\B^T }
\]
has full row rank, which is equivalent to the matrix
\[
  \pmat{  h\lld & h\Bl^T & a_l^T \\[1pt]
          h\Bld & H\BBd  & A\B^T }
\]
having full row rank by Proposition~\ref{propA-diag},

Again, the search direction is unique and (\ref{eqn-primaldir}) implies
that $\tmattt{ a_l }{ a_k }{ A\B }{ -M }$ has full row rank. As
$\dxtilde_k\ne 0$, Proposition~\ref{propA-AMfullrowrank} implies that
$\tmatt{ a_l }{ A\B }{ -M }$ must have full row rank. Consequently,
Proposition~\ref{propA-nonsing} implies that $K_l$ is nonsingular.

As $\dxtilde_k$, $\dxtilde_l$, $\dztilde_k$ and $\dztilde_l$ are all
nonzero, the roles of $k$ and $l$ may be reversed to give the result that
$K_k$ is nonsingular.
\end{proof}

The next proposition concerns the situation when a constraint index $k$
is moved from $\setN$ to $\setB$ in a dual subiteration. In particular, it
is shown that the resulting matrix $K\B$ defined after a subiteration is
nonsingular.

\begin{proposition}\label{propA-dualnonsing}
Assume that there is a unique solution of the equations
\begin{equation}\label{eqn-dualdir}
  \pmat{ h\lld & h\kld & h\Bl^T & a_l^T & \m 1        \\[1pt]
         h\kld & h\kkd & h\Bk^T & a_k^T &      & \m 1 \\[1pt]
         h\Bld & h\Bkd & H\BBd  & A\B^T &             \\[1pt]
         a_l   &  a_k  & A\B    &  -M   &             \\[1pt]
         1     &       &        &       &   -1        \\[1pt]
               &  1                                     }
  \pmat{ \m\dx_l \\
         \m\dx_k \\
         \m\dx\B \\
          -\dy   \\
          -\dz_l \\
          -\dz_k }
= \pmat{ 0 \\[1pt]
         0 \\[1pt]
         0 \\[1pt]
         0 \\[1pt]
         1 \\[1pt]
         0   },
\end{equation}
with $\dz_k\ne 0$. Then, the matrices $K_l$ and $K_k$ are nonsingular,
where
\[
 K_l = \pmat{ h\lld  & h\Bl^T & a_l^T \\[1pt]
              h\Bld  & H\BBd  & A\B^T \\[1pt]
              a_l    & A\B    &  -M     },  \wordss{and}
 K_k = \pmat{ h\kkd  & h\Bk^T & a_k^T \\[1pt]
              h\Bkd  & H\BBd  & A\B^T \\[1pt]
              a_k    & A\B    &  -M     }.
\]
\end{proposition}

\begin{proof}
As (\ref{eqn-dualdir}) has a unique solution with $\dz_k\ne 0$,
there is no solution of
\begin{equation}\label{eqn-basedual}
  \pmat{ h\lld & h\kld & h\Bl^T & a_l^T & \m 1 \\[1pt]
         h\kld & h\kkd & h\Bk^T & a_k^T &      \\[1pt]
         h\Bld & h\Bkd & H\BBd  & A\B^T &      \\[1pt]
         a_l   & a_k   & A\B    &  -M   &      \\[1pt]
         1     &       &        &       &   -1 \\[1pt]
               & 1                               }
  \pmat{ \m\dx_l \\
         \m\dx_k \\
         \m\dx\B \\
          -\dy   \\
          -\dz_l }
= \pmat{ 0 \\[1pt]
         0 \\[1pt]
         0 \\[1pt]
         0 \\[1pt]
         1 \\[1pt]
         0   }.
\end{equation}
The fundamental theorem of linear algebra applied to
(\ref{eqn-basedual}) implies the existence of a solution of
\begin{equation}\label{eqn-basedual-alt}
  \pmat{ h\lld & h\kld & h\Bl^T & a_l^T & \m1               \\[1pt]
         h\kld & h\kkd & h\Bk^T & a_k^T &    &\mbc{a_l^T}{1}\\[1pt]
         h\Bld & h\Bkd & H\BBd  & A\B^T &                   \\[1pt]
         a_l   & a_k   & A\B    &  -M   &                   \\[1pt]
         1     &       &        &       &  -1       }
  \pmat{ \m\dxtilde_l \\
         \m\dxtilde_k \\
         \m\dxtilde\B \\
          -\dytilde   \\
          -\dztilde_l \\
          -\dztilde_k  }
= \pmat{ 0 \\[1pt]
         0 \\[1pt]
         0 \\[1pt]
         0 \\[1pt]
         0   },
\end{equation}
with $\dztilde_l\ne0$. It follows from (\ref{eqn-basedual-alt}) that
$\dxtilde_l+\dztilde_l=0$. As $\dztilde_l\ne0$, this implies
$\dxtilde_l \dztilde_l<0$. The solution of (\ref{eqn-basedual-alt}) may be
regarded as a solution of the homogeneous equations $H \dx - A\T \dy - \dz
= 0$, $A \dx +  M \dy = 0$, with $\dz_i = 0$, for $i\in\setB$, and $\dx_i =
0$, for $i\in\{1,\dots,n\}\setminus\{k\}\setminus\{l\}$. Hence,
Proposition~\ref{prop-dotproduct} gives
\[
\dxtilde_l \dztilde_l + \dxtilde_k \dztilde_k \ge 0,
\]
so that $\dxtilde_k \dztilde_k > 0$. Hence, it must hold that
$\dxtilde_k\ne0$ and $\dztilde_k\ne 0$.

As $\dxtilde_k\ne0$, $\dxtilde_l\ne0$, $\dztilde_k\ne0$ and
$\dztilde_l\ne0$, the remainder of the proof is analogous to that of
Proposition~\ref{propA-primalnonsing}.
\end{proof}

The next result gives expressions for the primal and dual objective
functions in terms of the computed search directions.

\begin{proposition}\label{propA-dirs}
Assume that $(x,y,z)$ satisfies the primal and dual equality constraints
\[
 H x + c - A\T y - z = 0, \wordss{and}  Ax +   M y - b = 0.
\]
Consider the partition $\{1$, $2$, \dots, $n\}= \setB\cup\{l\}\cup\setN$
such that $x\N+q\N=0$ and $z\B+r\B=0$. If the components of the direction
$(\dx$, $\dy$, $\dz)$ satisfy {\rm(\ref{eqn-dxdydz})}, then the primal and
dual objective functions for {\rm$(\Primal_{q,r})$} and
{\rm$(\Dual_{q,r})$}, i.e.,
\begin{align*}
 f\P(x,y)   &= \m \half x\T H x + \half y\T  M y + c\T x + r\T x  \\
 f\D(x,y,z) &=  - \half x\T H x - \half y\T  M y + b\T y - q\T z,
\end{align*}
satisfy the identities
\begin{align*}
 f\P(x+\alpha\dx,y+\alpha\dy)
  &=  f\P(x,y)+ \dx_l(z_l+r_l) \alpha + \half \dx_l \dz_l \alpha^2, \\
 f\D(x+\alpha\dx,y+\alpha\dy,z+\alpha\dz)
  &=  f\D(x,y,z) - \dz_l(x_l+q_l) \alpha - \half \dx_l \dz_l \alpha^2.
\end{align*}
\end{proposition}

\begin{proof}
The directional derivative of the primal objective function is given by
\begin{subequations}\label{eqn-Pdirder}
\begin{align}
 \pmat{ \dx^T & \dy^T } \Grad f\P(x,y)
     &= \pmat{ \dx^T & \dy^T }
        \pmat{  Hx + c + r \\
                  M y        }            \notag              \\
     &= \pmat{ \dx^T & \dy^T }
        \pmat{ A\T y + z + r \\
                  M y          }          \label{eqn-Pdirderb}\\
     &= (A\dx+  M\dy)\T y + \dx\T ( z + r)
      =  \dx_l ( z_l + r_l ),             \label{eqn-Pdirderd}
\end{align}
\end{subequations}
where the identity $Hx+c = A\T y +z$ has been used in (\ref{eqn-Pdirderb})
and the identities $A\dx + M\dy=0$, $\dx_N=0$ and $z\B+r\B=0$ have
been used in (\ref{eqn-Pdirderd}).

The curvature in the direction $(\dx,\dy)$ is given by
\begin{equation}\label{eqn-Pdircurv}
  \pmat{ \dx^T & \dy^T } \Hess f\P(x,y) \pmat{ \dx \\ \dy }
= \pmat{ \dx^T & \dy^T }
    \pmat{   H &   \\
               & M   }
    \pmat{ \dx \\ \dy }
=  \dx_l \dz_l,
\end{equation}
where the last equality follows from Proposition~\ref{prop-dotproduct}.

The directional derivative  of the dual objective function is given by
\begin{subequations}\label{eqn-Ddirder}
\begin{align}
  \pmat{ \dx^T & \dy^T & \dz^T } \Grad f\D(x,y,z)
  &= \pmat{ \dx^T & \dy^T & \dz^T }
    \pmat{  -Hx     \\
            -My + b \\
            -q }    \\
 &= - \dx\T H x + \dy\T (-My + b) - \dz^T q \\
 &= - (A\T \dy + \dz)\T x + \dy\T (-My + b)
  - \dz^T q             \label{eqn-Ddirderc} \\
 &= -\dy\T (Ax +  My - b)- \dz\T (x+q) \\
 &= - \dz_l (x_l+q_l),  \label{eqn-Ddirdere}
\end{align}
\end{subequations}
where the identity $H \dx-A\T \dy -\dz=0$ has been used in (\ref{eqn-Ddirderc})
and the identities $Ax + My - b = 0$, $x\N+q\N=0$ and $\dz\B=0$ have
been used in (\ref{eqn-Ddirdere}).

As $z$ only appears linearly in the dual objective function, it follows
from the structure of the Hessian matrices of $f\P(x,y)$ and $f\D(x,y,z)$
in combination with (\ref{eqn-Pdircurv}) that
\begin{align*}
 \pmat{ \dx^T & \dy^T & \dz^T} \Hess f\D(x,y,z)
   \pmat{ \dx \\ \dy \\ \dz }
 &= -\pmat{ \dx^T & \dy^T}\Hess f\P(x,y) \pmat{ \dx \\ \dy } \\
 &= -\dx_l \dz_l.
\end{align*}
\end{proof}

The final result shows that there is no loss of generality in assuming that
$\tmat{ A }{ M }$ has full row rank in $(\Primal_{q,r})$.

\begin{proposition}\label{prop-AMfullrowrank}
There is no loss of generality in assuming that $\tmat{ A }{ M }$ has full row
rank in {\rm$(\Primal_{q,r})$}.
\end{proposition}

\begin{proof}
  Let $(x$, $y$, $z)$ be any vector satisfying (\ref{eqn-optgradLzero}) and
  (\ref{eqn-optfeasprimallin}).  Assume that $\tmat{ A }{ M }$ has linearly
  dependent rows, and that $\tmat{ A }{ M }$ and $b$ may be partitioned
  conformally such that
\[
 \mat{A}{M}
  = \pmat{A_1      &  M_{11}   &  M_{12}    \\[1pt]
          A_2\drop &  M_{12}^T &  M_{22}\drop }, \wordss{and}
 b
  = \pmat{ b_1 \\ b_2 },
\]
with $\tmatt{ A_{1} }{ M_{11} }{  M_{12} }$ having full row rank, and
\begin{equation}\label{eqn-AMpart}
 \pmat{ A_2\drop &  M_{12}^T &  M_{22}\drop }
  = N \pmat{ A_1 &  M_{11} &  M_{12}},
\end{equation}
with $A_1\in\Re^{m_1\times n}$ and $A_2\in\Re^{m_2\times n}$ for some
matrix $N\in\Re^{m_2\times m_1}$. From the  linear dependence of the
rows of $\tmat{ A }{ M }$, it follows that $x$, $y$ and $z$ satisfy
(\ref{eqn-optgradLzero}) and (\ref{eqn-optfeasprimallin}) if and only
if
\begin{align*}
 Hx+c - A_1^T y_1\drop - A_2^T y_2\drop  - z  &= 0, \\
 A_1 x +  M_{11} y_1   +  M_{12} y_2     - b_1&= 0 \wordss{and}  b_2 = N b_1.
\end{align*}
It follows from (\ref{eqn-AMpart}) that $M_{12} =  M_{11} N^T$ and
$A_2^T=A_1^T N^T$, so that $x$, $y$ and $z$ satisfy
(\ref{eqn-optgradLzero}) and (\ref{eqn-optfeasprimallin}) if and only if
\begin{align*}
    Hx + c - A_1^T( y_1\drop + N\T y_2\drop) - z   &= 0, \\
 A_1 x +         M_{11}( y_1 + N\T y_2)      - b_1 &= 0 \words{and} b_2 = N b_1.
\end{align*}
We may now define $\ytilde_1=y_1+N\T y_2$ and replace
(\ref{eqn-optfeasprimallin}) and (\ref{eqn-optgradLzero}) by the system
\begin{align*}
  H x + c - A_1^T\ytilde_1\drop - z   &= 0, \\
A_1 x +    M_{11}\ytilde_1      - b_1 &= 0.
\end{align*}
By assumption, $\tmatt{A_{1} }{ M_{11} }{ M_{12} }$ has full row
rank. Proposition~\ref{propA-diag} implies that $\tmat{A_{1} }{ M_{11} }$ has
full row rank. This gives an equivalent problem for which
$\tmat{ A_{1} }{ M_{11} }$ has full row rank.
\end{proof}

\end{document}